\documentclass[oneside,reqno]{amsart}
\usepackage[english]{babel}
\usepackage{amsfonts}
\usepackage{amsmath}
\usepackage{amssymb}
\usepackage{amstext}
\usepackage{amsxtra}
\usepackage{cancel}
\usepackage{enumitem}
\usepackage{graphicx}
\usepackage{psfrag}
\usepackage{ulem}
\usepackage{xcolor}
\usepackage[colorlinks,citecolor=blue,urlcolor=blue,bookmarks=true]{hyperref}
\newcommand{\A}{\mathcal A}
\newcommand{\aC}{\apartment_{\ssf\mathbb{C}}}
\newcommand{\apartment}{\mathfrak{a}}
\newcommand{\Atildetwo}{\smash{\widetilde{A}_{\ssf 2}}}
\newcommand{\Atilder}{\smash{\widetilde{A}_{\ssf r}}}
\newcommand{\boldpi}{\boldsymbol{\pi}}
\newcommand{\boldsigma}{\boldsymbol{\sigma}}
\newcommand{\C}{\mathbb{C}}
\newcommand{\cl}[1]{\operatorname{cl}\ssf({#1})}
\newcommand{\Czero}{C_{\ssf0}}
\newcommand{\cone}{c_{\hspace{.15mm}1}}
\newcommand{\ctwo}{c_{\ssf2}}
\newcommand{\const}{\operatorname{const.}}
\renewcommand{\epsilon}{\varepsilon}
\newcommand{\msb}{\hspace{-.4mm}}
\newcommand{\msf}{\hspace{.4mm}}
\newcommand{\N}{\mathbb{N}}
\renewcommand{\Im}{\operatorname{Im}}
\renewcommand{\O}{\mathcal{O}}
\newcommand{\R}{\mathbb{R}}
\renewcommand{\Re}{\operatorname{Re}}
\newcommand{\sector}{\apartment^+}
\newcommand{\ssb}{\hspace{-.25mm}}
\newcommand{\ssf}{\hspace{.25mm}}
\newcommand{\X}{\mathcal{X}}
\newcommand{\vsb}{\hspace{-.1mm}}
\newcommand{\vsf}{\hspace{.1mm}}
\newcommand{\Z}{\mathbb{Z}}

\newtheorem{theorem}{Theorem}[section]
\newtheorem{proposition}[theorem]{Proposition}
\newtheorem{lemma}[theorem]{Lemma}
\newtheorem{corollary}[theorem]{Corollary}
\newtheorem{remark}[theorem]{Remark}

\numberwithin{equation}{section}

\title[Distinguished random walks on affine buildings of type $\Atilder$]{Sharp estimates for distinguished random walks\\on affine buildings of type $\Atilder$}
\author[Anker]
{Jean-Philippe Anker}
\address{Universit\'e d'Orl\'eans, Universit\'e de Tours \& CNRS,
Institut Denis Poisson (UMR 7013),
B\^atiment de Ma\-th\'e\-ma\-ti\-ques,
B.P.~6759, 45067 Orl\'eans cedex 2, France}
\email{anker@univ-orleans.fr}
\author[Schapira]
{Bruno Schapira}
\address{Aix-Marseille Universit\'e \& CNRS,
Institut de Math\'ematiques de Marseille (I2M, UMR 7373),
Groupe Math\'ematiques de l'Al\'eatoire (ALEA), Equipe Probabilit\'es
\linebreak
(PROBA),
Centre de Math\'ematiques et Informatique,
Technop\^ole Ch\^ateau-Gombert,
39 rue Fr\'ed\'eric Joliot-Curie,
13453 Marseille Cedex 13, France}
\email{bruno.schapira@univ-amu.fr}
\author[Trojan]
{Bartosz Trojan}
\address{Politechnika Wroc{\l}awska,
Wydzia{\l} Matematyki,
Wyb.~Wyspia{\'n}skiego 27,
50-370 Wroc{\l}aw, Polska}
\email{bartosz.trojan@gmail.com}
\date{\today}
\subjclass[2010]
{Primary\,:
33C52,
51E24,
60B15,
60J10.
Secondary\,:
20E42,
20F55,
22E35,
31C35,
43A90,
60G50,
60J45.}
\keywords
{Affine building,
random walk,
transition density,
heat kernel,
Green function,
global estimate.}

\begin{document}

\begin{abstract}{We study a distinguished random walk
on affine buildings of type \ssf$\smash{\Atilder}$\ssf,
which was already considered by Cartwright, Saloff-Coste and Woess.
In rank $r\ssb=\ssb2$\ssf,
it is the simple random walk and we obtain optimal global bounds for its transition density
(same upper and lower bound, up to multiplicative constants).
In the higher rank case, we obtain sharp uniform bounds in fairly large space-time regions which are sufficient for most applications.}
\end{abstract}

\maketitle

\centerline{\it
In memory of Gerrit van Dijk (1939-2022)
}

\section{Introduction}

Heat diffusion in the continuous setting, respectively random walks in the discrete setting are extensively studied.
A main issue consists in obtaining sharp upper and/or lower bounds for the heat kernel in the continuous setting,
respectively for transition densities of random walks in the discrete setting.
Actually there are few cases where the same global bound (up to multiplicative constants) have been obtained.
Apart from the Euclidean setting, this was achieved for instance for Riemannian symmetric spaces of noncompact type \cite{AJ,AO}.
\smallskip

This work started as an attempt to obtain similar results for isotropic random walks on affine buildings\ssf;
it remained unpublished \cite{AST}
but we believe that the techniques as well as the result may be interesting for the mathematical community.
More precisely, we study a distinguished random walk on affine buildings of type $\smash{\widetilde{A}_r}$.
This walk was already considered in \cite{SW, CW}\,;
it is the simple random walk in rank $r\msb\le\msb2$ but not in rank $r\msb>\msb2\ssf$.
In rank $r\msb=\msb1$\vsf, affine buildings are homogeneous trees
and the simple random walk was already studied in \cite{L1,L2,W}.
In rank \ssf$r\ssb=\ssb2$\ssf, we obtain the same global upper and lower bound, up to multiplicative constants.
In higher rank \ssf$r\ssb>\ssb2$\ssf, we obtain the same result in a fairly large space-time region, as well as a global upper bound,
which are sufficient for applications, for instance to deduce a global upper and lower bound for the Green function (see \cite{T}).
We recover in particular the spectral gap computed in \cite{SW} and the local limit theorem established in \cite{CW,P3}.
\smallskip

In the meantime, the third author was able to generalize the latter results
to finitely supported isotropic random walks on any affine building \cite{T}.
Nonetheless we consider that this work is worth publishing,
since our results are global for buildings of type $\widetilde{A}_2$
and since our present approach is much simpler than the delicate analysis carried out in \cite{T}.
Our initial aim consisted in analyzing the behavior of isotropic random walks on affine buildings,
in order to describe the Martin compactification of affine buildings.
This was eventually achieved in \cite{RT}.
\smallskip

Our paper is organized as follows.
Basics are recalled in Section \ref{basics}.
Section \ref{sectionA2} is devoted to our main result,
namely a global upper and lower bound for the simple random walk on affine buildings of type $\widetilde{A}_2$.
At the end of Section \ref{sectionA2} and in Section \ref{sectionAr},
we extend in part our results to isotropic nearest neighbor random walks in rank $r\msb=\msb2$
and to to the distinguished random walk in rank $r\msb>\msb2\ssf$.
Appendix \ref{Appendix}  is devoted to some remarkable formulae
for the Fourier transform of any nearest neighbor random walk in rank $r\msb=\msb2$\ssf.
\smallskip

Our method relies on a careful analysis of the transition density using the inverse Fourier transform.
Fourier analysis on $p$-adic like buildings goes back to Macdonald \cite{M1} in the early seventies.
There was no other reference available in the literature for a long time,
until Cartwright \cite{C} and Parkinson \cite{P1, P2} on the one hand,
Mantero and Zappa \cite{MZ1, MZ2} on the other hand,
resumed Fourier analysis on affine buildings.

Our work, as well as all aforementioned ones about buildings, deals with vertices
and it is natural to consider higher dimension simplices in affine buildings.
A first study of random walks on chambers in affine buildings of type $\Atildetwo$ 
is carried out in \cite{PS}, relying on the theory developed in \cite{O}.

\section{Preliminaries}\label{basics}

\subsection{Notation}

Throughtout the paper,
\ssf$\N$ \ssf will denote the set of nonnegative integers \ssf$\{\ssf0,1,2,\ssf\dots\msf\}$ \ssf
and \ssf$\N^*$ the set of positive integers \ssf$\{\vsf1,2,\ssf\dots\msf\}$\ssf.
The maximum of two real numbers $a$ and $b$ \ssf will be denoted by \ssf$a\msb\vee\msb b$\ssf,
and the minimum by \ssf$a\ssb\wedge\ssb b\ssf$.
Let us finally specify the meaning of the following binary symbols
between two nonnegative functions $f$ and $g$\,:
\begin{description}[labelindent=4pt,labelwidth=4mm,labelsep*=1pt,leftmargin=!]
\item[\textbullet]
$f\ssb\lesssim\ssb g$\ssf, resp.~$f\ssb\gtrsim\ssb g$ \ssf means that
there exists a constant \ssf$C\!>\!0$ \ssf such that
$f\ssb\le\ssb C\ssf g$\ssf,  resp.~$f\ssb\ge\ssb C\ssf g$\ssf,
\item[\textbullet]
$f\ssb\approx\ssb g$ \ssf stands for
\ssf$f\ssb\lesssim\ssb g$ \ssf and \ssf$f\ssb\gtrsim\ssb g$\ssf,
\item[\textbullet]
in the case of positive functions,
$f\ssb\sim\ssb g$ \ssf means that the ratio $\smash{\frac fg}$ tends to 1.
\end{description}

\subsection{Root system}

We recall some standard notation and refer e.g.~to \cite{B} for more details.
In the Euclidean space \ssf$\R^{r+1}$, consider the subspace
\vspace{1mm}

\centerline{$
\apartment=\{\ssf z\!\in\!\R^{r+1}\msf|\,z_1\!+\vsf\dots\vsf+\ssb z_{\vsf r+1}\!=\msb0\ssf\}
$}\vspace{1mm}

and the root system of type $A_{\vsf r}$
\vspace{1mm}

\centerline{$
R=\{\ssf e_j\!-\msb e_k\msf|\msf1\!\le\ssb j\msb\ne\ssb k\ssb\le\ssb r\msb+\!1\vsf\}\msf.
$}

Notice that
\begin{description}[labelindent=4pt,labelwidth=4mm,labelsep*=1pt,leftmargin=!]
\item[\textbullet]
the inner product \ssf$\langle\ssf\cdot\ssf,\ssf\cdot\ssf\rangle$
\ssf on \ssf$\R^{r+1}$ extends to a $\C$-bilinear form on \ssf$\C^{\vsf r+1}$
and in particular on the complexification \ssf$\aC\ssb=\apartment\ssb+\ssb i\ssf\apartment$\ssf,
\item[\textbullet]
each root \ssf$\alpha\!\in\!R$ \ssf coincides with its coroot
\ssf$\alpha^\vee\hspace{-.8mm}=\!\frac2{\|\alpha\|^2}\ssf\alpha$\ssf.
\end{description}
Consider the positive subsystem
\vspace{1mm}

\centerline{$
R^+\ssb=\{\ssf e_j\!-\msb e_k\msf|\msf1\!\le\ssb j\msb<\ssb k\ssb\le\ssb r\msb+\!1\vsf\}
$}\vspace{1mm}

in \ssf$R$\ssf, the corresponding positive Weyl sector
\vspace{1mm}

\centerline{$
\sector\ssb=\{\ssf z\!\in\!\apartment\,|\,z_1\!>\vsf\dots\vsf>\msb z_{\vsf r+1}\ssf\}
$}

in \ssf$\apartment$ \ssf and its closure

\centerline{$
\cl{\sector}\ssb=\{\ssf z\!\in\!\apartment\,|\,z_1\!\ge\vsf\dots\vsf\ge\msb z_{\vsf r+1}\ssf\}\msf,
$}

the basis of simple roots

\centerline{$
\alpha_j\msb=\ssb e_j\msb-\ssb e_{j+1}
\qquad(1\!\le\ssb j\msb\le\ssb r)
$}\vspace{1mm}

and the dual basis of fundamental weights
\vspace{1mm}

\centerline{$\displaystyle
\lambda_j=\sum\nolimits_{\vsf1\le k\le j}\hspace{0mm}e_k
-\tfrac j{r+1}\sum\nolimits_{\vsf1\le k\le r+1}\hspace{0mm}e_k
\qquad(1\!\le\ssb j\msb\le\ssb r)\msf.
$}

Notice that

\centerline{$
\sum_{\ssf j=1}^{\,r} \lambda_j
\quad\text{is equal to}\quad
\rho\,=\ssf\frac12\,\sum_{\ssf\alpha \in R^+}\ssb\alpha\;.
$}\vspace{1mm}

Let \ssf$P\ssb=\ssb\sum_{\ssf j=1}^{\,r}\Z\ssf\lambda_j$ be the weight lattice
and \ssf$Q\ssb=\ssb\sum_{\ssf j=1}^{\,r}\Z\ssf\alpha_j$ the root sublattice.
Let \ssf$P^+\hspace{-.5mm}=\ssb P\cap\smash{\cl{\sector}}\hspace{-.5mm}
=\ssb\smash{\sum_{\ssf j=1}^{\,r}}\N\ssf\lambda_j$,
resp.~\ssf$P^{++}\hspace{-.5mm}=\ssb P\cap\sector\hspace{-.5mm}
=\ssb\smash{\sum_{\ssf j=1}^{\,r}}\N^*\lambda_j$,
be the subset of dominant weights, resp.~strictly dominant weights.
The length of \ssf$\lambda\!\in\!P^+$ is \ssf$|\lambda|\ssb
=\ssb\sum_{\ssf j=1}^{\,r}\langle\ssf\alpha_j,\lambda\ssf\rangle$\ssf.
The Coxeter complex is a simplicial complex in \ssf$\apartment$
\ssf whose set of vertices is \ssf$P$.
Its maximal simplices are called chambers.
The fundamental chamber \ssf$\Czero$ has vertices
$\lambda_{\ssf0}\!=\!0\ssf, \lambda_1,\dots,\lambda_{\ssf r}$\ssf.
Define the label function
\,$\tau\ssb:\ssb P\ssb\longrightarrow\ssb\{0,\dots,r\}$
by the following two conditions\,:
\begin{description}[labelindent=4pt,labelwidth=4mm,labelsep*=1pt,leftmargin=!]
\item[\textbullet]
For the fundamental chamber, \ssf$\tau(\lambda_j)\!=\!j$\ssf.
\item[\textbullet]
For each chamber, the labels of vertices are exactly $0\ssf,\dots,r$\ssf.
\end{description}
Let \ssf$W_0\msb=\ssb\mathfrak{S}_{r+1}$ be the Weyl group,
generated by the reflections
\ssf$r_{\!\alpha}(z)\!=\!z\!-\!\langle\alpha, z\rangle\ssf\alpha$
\ssf along roots $\alpha\!\in\!R$\ssf,
and let $W\hspace{-1mm}=\!W_0\!\ltimes\!Q$,
resp.~$\widetilde W\hspace{-1mm}=\!W_0\!\ltimes\!P$
be the affine Weyl group, resp.~the extended Weyl group,
generated by $W_0$
and by the translations \ssf$\tau_\lambda$
along \ssf$\lambda\!\in\!Q$, resp.~$\lambda\!\in\!P$.
Stabilizers will be denoted by subscripts,
for instance $W_0\!=\!\widetilde W_0$
is the stabilizer of $0$ in $W$ \!and in $\widetilde W$\ssb.
Let \,$w\ssb\longmapsto\ssb q_{\ssf w}$
\ssf be an extended parameter function on \ssf$\widetilde W$\ssb.
For the type \ssf$\Atilder$ considered in this work,
recall that there is an integer \ssf$q\!>\!1$ such that
\begin{equation*}
q_{\ssf w}=\ssf q^{\,\ell\ssf(w)}
\qquad\forall\,w\in W_0\ssf,
\end{equation*}
where \ssf$\ell\ssf(w)$ denotes the lenght of \ssf$w$ in $W_0$\ssf, and
\begin{equation*}
q_{\ssf t_\lambda}\!=q^{\ssf2\ssf\langle\rho,\ssf\lambda\rangle}
\qquad\forall\;\lambda\!\in\!P^+.
\end{equation*}
Eventually the following Poincar\'e polynomial
\begin{equation*}
V(q^{-1})=\sum\nolimits_{\vsf w\in V}q_{\ssf w}^{-1}
\end{equation*}
is attached to every subset \ssf$V$ \hspace{-.5mm}of \ssf$W_0$\ssf.

\begin{figure}[ht]
\begin{center}
\psfrag{a+}[c]{\color{red}$\sector$}
\psfrag{alpha1}[c]{\color{red}$\alpha_1$}
\psfrag{alpha2}[c]{\color{red}$\alpha_2$}
\psfrag{alpha1=0}[c]{\color{red}$\alpha_1\!=\ssb0$}
\psfrag{alpha2=0}[c]{\color{red}$\alpha_2\ssb=\ssb0$}
\psfrag{alpha1=alpha2}[c]{\color{red}$\alpha_1\!=\ssb\alpha_2$}
\psfrag{C}[c]{\color{red}$C_0$}
\psfrag{lambda1}[c]{\color{red}$\lambda_1$}
\psfrag{lambda2}[c]{\color{red}$\lambda_2$}
\psfrag{rho}[c]{\color{red}$\rho$}
\includegraphics[width=100mm]{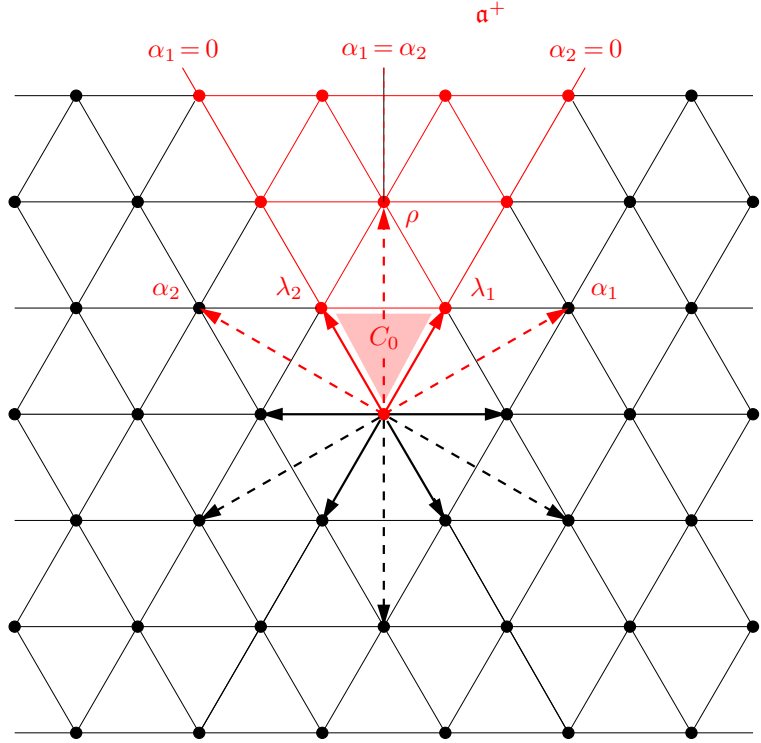}
\end{center}
\caption{Apartment in an affine building of type $\Atildetwo$}
\end{figure}

\subsection{Affine building}

In this subsection, we follow mostly \cite{R}
and refer for more details to \cite{C,P1}.
An affine building of type \ssf$\Atilder$
is a nonempty simplicial complex
containing subcomplexes called apartments such that:
\begin{description}[labelindent=4pt,labelwidth=4mm,labelsep*=1pt,leftmargin=!]
\item[\textbullet]
Each apartment is isomorphic to the Coxeter complex
associated to \ssf$\Atilder$\ssf.
\item[\textbullet]
Any pair of simplices is contained in an apartment.
\item[\textbullet]
Given two apartments which share at least one chamber (simplex of maximal dimension),
there exists a unique isomorphism between them,
which fixes pointwise their intersection.
\end{description}
The building will be assumed to be regular, thick and locally finite.
By definition this means that,
given any chamber \ssf$C$
and any face \ssf$F$ (simplex of codimension 1) of \ssf$C$,
the cardinality of the set of chambers
different from \ssf$C$ and containing \ssf$F$
is independent of \ssf$C$ and \ssf$F$,
and is equal to \ssf$2\ssb\le\ssb q\ssb<\msb\infty$\ssf.
We denote by $\X$ the set of vertices
(simplices of dimension $0$) of the building.
In rank \ssf$r\!=\!1$\ssf,
buildings are homogeneous trees with \ssf$q\ssb+\!1$ \ssf edges.
Fix a base point \ssf$0\!\in\!\X$.
It follows from the above definition that
one can define a label function \,$\tau\ssb:\ssb\X\to\{\ssf0,\dots,r\ssf\}$
\,such that \ssf$\tau(0)\ssb=\ssb0$ \ssf
and all isomorphisms in the definition preserve labels.
Now given \ssf$x\!\in\!\X$,
there exists an apartment \ssf$\smash{\widetilde{P}}$
\ssf containing \ssf$0$ \ssf and \ssf$x$ \ssf
and a label preserving isomorphism
between \ssf$\smash{\widetilde{P}}$ \ssf and \ssf$P$,
sending \ssf$0$ \ssf to \ssf$0$ \ssf
and \ssf$x$ \ssf to an element of \ssf$P^+$\ssb.
The image of \ssf$x$ \ssf by such an isomorphism
is called the radial part of \ssf$x$ \ssf
and will be denoted by \ssf$x^+$.
Let
\begin{equation*}
V_\lambda\ssb(0)
=\{\,y\!\in\!\X\mid y^+\hspace{-1mm}=\!\lambda\,\},
\end{equation*}
be the so-called sphere of radius \ssf$\lambda\!\in\!P^+$
centered at $0$\ssf.
For every $y\!\in\!V_\lambda\ssb(0)$,
we set $|y|\!=\!|\lambda|$ and $y_j\!=\!\langle\alpha_j,\lambda\rangle$.
More generally, the sphere $V_\lambda(x)$
of radius $\lambda\!\in\!P^+$ and center $x\!\in\!\X$
consists of all \ssf$y\!\in\!\X$
such that there exist an apartment $\smash{\widetilde P}$ containing $\{x,y\}$
and an isomorphism between $\smash{\widetilde P}$ and $P$
which preserves labels up to translation 
and which sends \ssf$x$ to $0$ and \ssf$y$ to $\lambda$\ssf.
The cardinality of $V_\lambda\ssb(x)$
is independent of \ssf$x$ and is given by
\begin{equation*}
N_\lambda=\ssf\tfrac{W_0(q^{-1})}
{(W_0\ssf\cap\ssf W_{\ssb\lambda})(q^{-1})}\,
q^{\ssf2\ssf\langle\rho,\ssf\lambda\rangle}\ssf.
\end{equation*}
Finally $(x,y)^+\!=\ssb\lambda\!\in\!P^+$ and $d(x,y)\ssb=\ssb|\lambda|\!\in\![\ssf0,+\infty)$
are the vectorial and scalar distances between $x$ and $y$. Both are invariant under isomorphisms of $\X$.

\subsection{Special functions}
\label{SpecialFunctions}

Consider the fundamental skew invariant polynomial
\begin{equation*}
\boldpi(z)\ssf=\,\prod\nolimits_{\ssf\alpha\in R^+}\langle\ssf\alpha\ssf,\ssf z\ssf\rangle\,,
\end{equation*}
the Weyl denominator
\begin{equation*}
\Delta(z)
=\sum\nolimits_{\vsf w\in W_0}(\det w)\,e^{\,\langle w.\rho,z\rangle}
=\prod\nolimits_{\ssf\alpha \in R^+}
\bigl(\,e^{\frac{\langle\alpha,z\rangle}2}\!
-e^{-\frac{\langle\alpha,z\rangle}2}\ssf\bigr)\,,
\end{equation*}
and the following functions on $\aC$\,:
\begin{equation*}
\mathbf{b}(z)=\prod\nolimits_{\ssf\alpha\in R^+}\bigl(\,1\!-\ssb q^{-1}\ssf e^{-\langle\alpha,z\rangle}\ssf\bigr)\,,
\end{equation*}\vspace{-3.5mm}
\begin{equation}\label{cFunction}
\mathbf{c}(z)=e^{\ssf\langle\rho,z\rangle}\,{\Delta(z)}^{-1}\,\mathbf{b}(z)\ssf
=\prod\nolimits_{\ssf\alpha\in R^+}\ssb\tfrac{1\,-\,q^{-1}e^{-\langle\alpha,z\rangle}}{1\,-\,e^{-\langle\alpha,z\rangle}}\,,
\end{equation}\vspace{-3.5mm}
\begin{equation*}
h(z)=\sum\nolimits_{\ssf j=1}^{\,r}\sum\nolimits_{\ssf\lambda\in W_0.\lambda_j}\ssb e^{\ssf\langle\lambda,z\rangle}\,.
\end{equation*}
Notice that \ssf$\mathbf{b}(z)$ \ssf is bounded on
\ssf$\cl{\sector}\!+\ssb i\ssf\apartment$\ssf,
as well as all its derivatives,
and moreover that \ssf$|\mathbf{b}(z)|$ is bounded there from below.
The function \ssf$h(z)$ \ssf is a linear combination
of symmetric Macdonald functions
\begin{equation}\label{MacdonaldPolynomial}
P_\lambda(z)
=W_0(q^{-1})^{-1}\,q_{\ssf t_\lambda}^{-1/2}\,
\sum\nolimits_{\vsf w\in W_0}\mathbf{c}(w.z)\,e^{\langle\ssf\lambda,\ssf w.z\ssf\rangle}
\end{equation}
(see \cite{M4}, \cite{P2}).
For the type \ssf$\Atilder$ considered in this work,
there is indeed a close connection, namely
\vspace{-4mm}
\begin{equation*}
P_\lambda(z)\ssf
=\ssf\overbrace{\textstyle\frac{(W_0\ssf\cap\ssf W_{\ssb\lambda})(q^{-1})}{W_0(q^{-1})}\,
q^{\ssf-\langle\rho,\ssf\lambda\rangle}}^{N_\lambda^{-1}\,q^{\ssf\langle\rho,\lambda\rangle}}
P_\lambda(e^{\ssf z}\ssf;q^{-1})
\end{equation*}
between symmetric Macdonald polynomials $P_\lambda$
and Hall-Littlewood polynomials $P_\lambda(\,.\,;q^{-1})$,
which boil down to symmetric monomials
when $\lambda$ is a fundamental weight
(see for instance
\cite[pp. 209 \& 299]{M2},
\cite[\S\;10]{M3},
\cite[pp. 99--100]{C}).
Thus
\begin{equation*}
P_{\lambda_j}(z)=N_{\lambda_j}^{-1}\,q^{\ssf\langle\rho,\ssf\lambda_j\rangle}\ssf
\sum\nolimits_{\vsf\lambda\in W_0.\lambda_j}\ssb e^{\ssf\langle\lambda,\ssf z\rangle}
\end{equation*}
\vspace{-5.5mm}

and
\vspace{-.5mm}
\begin{equation*}
h(z)=\sum\nolimits_{\ssf j=1}^{\,r}N_{\lambda_j}\,q^{\ssf-\langle\rho,\ssf\lambda_j\rangle}\ssf P_{\lambda_j}(z)\,.
\end{equation*}
Now the {\it fundamental spherical function} is defined on \ssf$\X$ \ssf by
\begin{equation*}
F_0(x)\ssf=\,P_\lambda(0)
\quad\text{with}\quad\lambda\!=\!x^+.
\end{equation*}
It is a positive eigenfunction of the Hecke algebra $\A$
described in Subsection \ref{Hecke} below.
Specifically, for every $\mu\!\in\!P^+$,
there exists a constant $c_\mu\!>\!0$ such that
\begin{equation}\label{F0eigenfunction}
\sum\nolimits_{y\in V_\mu(x)}\!F_0(y)=c_\mu\ssf F_0(x)\ssf,
\end{equation}
for all $x\!\in\!\X$ (see \cite[Theorem 3.22]{P2}).
Its behavior is given in the following proposition.
The statement and the proof are similar to the symmetric space case \cite{A},
which was generalized to the hypergeometric setting in \cite{S1}.
\begin{proposition}
We have
\begin{equation}
\label{F0global}
F_0(x)\,\approx\;q^{-\langle \rho,\lambda\rangle}\,
\prod\nolimits_{\ssf\alpha\in R^+}\bigl(\ssf1\!+\!\langle\alpha,\lambda\rangle\bigr)\ssf.
\end{equation}
Moreover,
\begin{equation}
\label{F0asymp}
F_0(x)\ssf\sim\,\const\,
\boldpi(\lambda)\;q^{-\langle \rho,\lambda\rangle},
\end{equation}
as \,$\langle\alpha,\lambda\rangle\ssb\to\ssb+\infty$
\,for all \,$\alpha\!\in\!R^+$.
\end{proposition}

\begin{proof}
Let us multiply \eqref{MacdonaldPolynomial} by $\boldpi(z)$\ssf,
in order to remove the singularity of the $\mathbf{c}$-function at the origin,
and apply the operator
$\boldpi\bigl(\frac\partial{\partial z}\bigr)\big|_{z=0}$\ssf.
We obtain in this way
\begin{equation*}
P_\lambda(0)=q^{-\langle \rho,\lambda\rangle}\,p(\lambda)\,,
\end{equation*}
where \ssf$p$ \ssf is a polynomial of the form
\begin{equation*}
p\,=\,\const\ssf\boldpi\,
+\,\text{linear combination of subproducts of \,}\boldpi\,.
\end{equation*}
This proves \eqref{F0asymp} and \eqref{F0global}
far away from the walls.
Eventually \eqref{F0eigenfunction} enables us
to extend \eqref{F0global} up to the walls,
since it implies that $F_0(x)\!\approx\!F_0(y)$
for all neighbors $x,y\!\in\!\X$
and more generally for all $x,y\!\in\!\X$
at any fixed distance.
\end{proof}

\subsection{Averaging operators and Fourier-Gelfand transform}
\label{Hecke}

We denote by $\A$ \ssf
the linear span of mean operators on \ssf$\X$\,:
\begin{equation*}
A_\lambda f(x)\ssf
=\,\tfrac{1\vphantom{|}}{N_\lambda\vphantom{\frac00}}\,
\sum\nolimits_{\ssf y\in V_\lambda\ssb(x)}f(y)
\qquad\forall\;x\!\in\!\X\ssf,\;\forall\;\lambda\!\in\!P^+.
\end{equation*}
Then $\A$ is a (commutative) polynomial $\ast$\ssf-algebra
with generators $A_{\lambda_1},\dots,A_{\lambda_r}$.
Consider $\A$ as a subalgebra of the algebra
$\mathcal L(\ell^{\ssf2}(\X))$
of bounded linear operators on $\ell^{\ssf2}(\X)$.
Then the closure \ssf$\overline{\A}$ \ssf
of \ssf$\A$ \ssf in \ssf$\mathcal{L}(\ell^{\ssf2}(\X))$ \ssf
is a commutative \ssf$C^\ast$\ssb-\ssf algebra
and the Fourier-Gelfand transform defines an isomorphism
between \ssf$\overline{\A}$ \ssf
and the algebra of continuous functions 
on \ssf$i\ssf\apartment$\ssf,
which are invariant under \ssf$W_0\!\ltimes\!i\ssf2\ssf\pi\ssf Q\ssf$.
Such functions can be viewed
as \ssf$W_0$\ssf-\ssf invariant continuous functions on \ssf$i\ssf U$\ssb,
where
\begin{equation*}
U\ssf=\,\{\,\theta\!\in\!\apartment\mid
|\langle\alpha,\ssb\theta\ssf\rangle|
\!\le\!2\ssf\pi\hspace{2mm}\forall\;\alpha\!\in \!R\,\}\ssf,
\end{equation*}
is a $W_0$\ssf-\ssf invariant fundamental domain
for the action of the lattice \ssf$2\ssf\pi\ssf Q$ on $\apartment$\ssf. 
Specifically, the image of \ssf$A_\lambda$ is
the Macdonald polynomial \ssf$P_\lambda\ssf$.
\smallskip

Finally the following inversion formula holds,
for every \ssf$A\!\in\!\overline{\A}$ and \ssf$x,\ssb y\!\in\!\X$
(see \cite[Theorem 5.1.5]{M1} or \cite[Theorem 5.2]{P2})\,:
\begin{equation}\label{InversionFormula}
(A\ssf\delta_x)(y)=\Czero\,{\displaystyle\int_{\ssf U}}\,
\widehat{A}\ssf(i\ssf\theta)\,P_\lambda(-\ssf i\ssf\theta)\;
\tfrac{d\hspace{.1mm}\theta}{|\ssf\mathbf{c}\ssf(i\ssf\theta)|^2}\,,
\end{equation}
\vspace{-2mm}

where \,$\Czero\ssb
=\ssb\frac{W_0(q^{-1})\vphantom{\frac00}}{(2\pi)^r\ssf|W_0|\vphantom{\frac00}}$
\ssf and \ssf$\lambda\ssb=\ssb (x,y)^+$\ssf.

\subsection{Distinguished random walk}

In this paper,
we consider mainly the Markov chain on $\X$
with transition probability
\begin{equation*}
p(x,y)=\,\begin{cases}
\,\sigma\,q_{\ssf t_{\lambda_j}}^{-1/2}
&\text{if \,}y\!\in\!V_{\lambda_j}\ssb(x),\\
\quad0
&\text{otherwise\ssf,}\\
\end{cases}
\end{equation*}
where
\begin{equation*}
\sigma^{-1}\,
=\,\sum\nolimits_{\ssf j=1}^{\,r}N_{\lambda_j}\,q_{\ssf t_{\lambda_j}}^{-1/2}.
\end{equation*}
This distinguished random walk,
which was already considered in \cite{SW} and \cite{CW},
is the {\it simple\/} random walk in rank one and two, but not in higher rank.
Its generator
\begin{equation}\label{DistinguishedRandomWalk}
Af(x)=\sum_{\ssf y\in\X}p(x,y)\,f(y)
=\sigma\,\sum_{\ssf j=1}^{\,r}
q_{\ssf t_{\lambda_j}}^{-1/2}\ssf N_{\lambda_j}\ssf A_{\lambda_j}f(x)\ssf,
\end{equation}
corresponds to \,$\widehat{A}\ssb=\ssb\sigma\ssf h$ \ssf via the Fourier-Gelfand transform
and its spectral radius is equal to \,$\boldsigma\ssb=\ssb\sigma\ssf h(0)$\ssf,
where \ssf$h(0)\ssb=\ssb\sum\nolimits_{\ssf j=1}^{\,r}\ssb|\vsf W_0\vsf.\ssf\lambda_j|\ssb
=\ssb2\ssf(2^{\vsf r}\hspace{-1mm}-\!1)$\vsf.

\section{Global heat kernel estimates in rank two}
\label{sectionA2}

For every integer \ssf$n\msb\ge\msb2$ \ssf and for every \ssf$x\!\in\!\X$
with \ssf$|x|\!<\!n$\ssf, set
\begin{equation*}
\delta
=\tfrac{x^+\ssb+\ssf\rho\vphantom{|}}{n\ssf+\ssf2}
=\underbrace{\tfrac{x_1+\ssf1\vphantom{|}}{n\ssf+\ssf2}}_{\delta_1}\lambda_1\ssb
+\underbrace{\tfrac{x_2+\ssf1\vphantom{|}}{n\ssf+\ssf2}}_{\delta_2}\lambda_2\ssf.
\end{equation*}
Notice that the coordinates \ssf$\delta_j$ belong to the interval \ssf$(0,1)$\ssf.
Thus \ssf$\delta\!\in\!\sector$
and \ssf$|\delta|\!=\!\smash{\frac{|x|\ssf+\ssf2}{n\ssf+\ssf2}}\!<\!1$\ssf.
This section is mainly devoted to the proof
of the following global estimate for the transition probabilities
\ssf$p_{\ssf n}(x)\!=\ssb p_{\ssf n}(0,x)$
\ssf of the random walk \eqref{DistinguishedRandomWalk}.

\begin{theorem}
\label{TheoremHeatEstimateA2}
Let \,$\phi(\delta)\ssb
=\min_{\ssf z\ssf\in\apartment\ssb}\Phi_\delta(z)$\ssf,
where \,$\Phi_\delta(z)\ssb
=\log\frac{h(z)}6\ssb-\ssb\langle\delta,z\rangle$\ssf.
Then
\begin{equation}\label{OptimalBoundA2}
p_{\ssf n}(x)\approx
\tfrac{(1\ssf+\ssf|x|\hspace{.1mm})\ssf
(1\ssf+\ssf x_1)\ssf(1\ssf+\ssf x_2)\vphantom{\big|}}
{n^{3\vphantom{\frac oo}}\ssf\sqrt{\ssf n\ssf-\ssf|x|\,}\ssf
\sqrt{\ssf n\ssf-\ssf x_1\ssb\vee x_2\ssf\vphantom{|}}}\;
\boldsigma^{\vsf n}\,q^{-\langle\rho,\ssf x^{\ssb+\ssb}\rangle}\ssf
e^{\ssf n\ssf\phi(\delta)}
\end{equation}
uniformly in the range \,$|x|\!<\!n$\ssf.
In the limit case \,$|x|\!=\!n$\ssf,
\begin{equation*}
p_{\ssf n}(x)\approx\ssf\sigma^{\ssf n}\,q^{-n}\,n^{\ssf n}\,
(\vsf x_1\hspace{-.75mm}\vee\hspace{-.5mm}x_2)^{-\ssf(x_1\ssb\vee x_2)}\,
(\vsf x_1\!\wedge\ssb x_2\!+\!1)^{-\ssf(x_1\wedge\ssf x_2)\ssf-\vsf\frac12}\ssf.
\end{equation*}
\end{theorem}

\begin{remark}
Let us comment on some factors occurring on the right hand side of \eqref{OptimalBoundA2}.
The exponential decay \,$\boldsigma^{\vsf n}$
is produced by the spectral radius \,$\boldsigma
=\ssb\smash{\frac{3\vphantom{|}}{q\ssf+1+\ssf q^{-1}}}\ssb
<\ssb1$ of \ssf$A$\ssf, \ssf$F_0(x)\ssb
\approx\ssb(1\!+\!|x|\hspace{.1mm})\ssf(1\!+\ssb x_1)\ssf(1\!+\ssb x_2)\,
q^{-\langle\rho,\ssf x^{\ssb+\ssb}\rangle}$
is the \ssf$W_0$-\ssf invariant ground state of \ssf$A$\ssf,
and \,$e^{\ssf n\ssf\phi(\delta)}\vphantom{\big|}$ is a Gaussian type factor.
\end{remark}

\begin{remark}
Recall \ssf$($\cite{L1}, see also \cite{W}$)$
the corresponding result in rank one i.e.~for homogeneous trees\,$:$
\begin{equation}\label{HeatEstimateA1}
p_{\ssf n}(x)\approx
\tfrac{1\ssf+\ssf|x|\vphantom{|}}{n\,\sqrt{1\ssf+\ssf n\ssf-\ssf|x|\ssf}}\;
\boldsigma^{\vsf n}\,
q^{-\frac{|x|}2}\ssf e^{\ssf n\ssf\phi(\delta)}
\quad\forall\;|x|\ssb\le\ssb n\text{ with same parities\ssf.}
\end{equation}
Here \,$\boldsigma\ssb
=\ssb\frac{2\vphantom{|}}{q^{1/2}\ssf+\,q^{-1/2}}$,
\ssf$\delta\ssb=\ssb\frac{|x|+\ssf1}{n\ssf+\ssf1}$
and
\begin{equation*}
\phi(\delta)\ssf=\ssf-\,\tfrac12\,
\{\ssf(1\hspace{-.75mm}+\!\delta)\log\ssf(1\hspace{-.75mm}+\!\delta)\ssb
+\ssb(1\!-\!\delta)\log\ssf(1\!-\!\delta)\ssf\}\ssf.
\end{equation*}
As \,$\lim_{\ssf\delta\nearrow1}\ssb\phi(\delta)\ssb=\ssb-\log 2$\ssf,
notice that \eqref{HeatEstimateA1} agrees with the obvious expression
\,$p_{\ssf n}(x)\!=\!(q\!+\hspace{-.75mm}1)^{-n}$
in the limit case \ssf$|x|\!=\!n$\ssf.
\end{remark}

Theorem \ref{TheoremHeatEstimateA2} is proved by combining different arguments,
depending on the relative sizes of \ssf$|x|$ \ssf and \ssf$n$
\ssf and on the position of \ssf$\lambda\ssb=\ssb x^+$
in \ssf$\smash{\cl{\sector}}$.
In most cases, the proof relies on suitable versions
of the inversion formula \eqref{InversionFormula},
which yields
\begin{equation}\label{HeatFormula1}
p_{\ssf n}(x)
=\Czero\,\sigma^{\ssf n}\ssb\int_{\ssf U}
h(i\ssf\theta)^n\,P_\lambda(-\ssf i\ssf\theta)\,
\tfrac{d\hspace{.1mm}\theta}{|\ssf\mathbf{c}\ssf(i\ssf\theta)|^2}\,,
\end{equation}
and thus boils down to estimating integrals of the form
\begin{equation}\label{Integrals}
\int_{\ssf U}\ssf
e^{\ssf(n+2)\ssf\Psi(\theta)}\,
a(\theta)\,d\ssf\theta\,,
\end{equation}
where \ssf$\Psi$ \ssf is a complex phase involving the function \ssf$h$
\ssf and \ssf$a$ \ssf is an amplitude involving
the functions \ssf$\mathbf{b}$ \ssf or \ssf$\mathbf{c}$\ssf.
Let us start the proof of Theorem \ref{TheoremHeatEstimateA2}
with a series of auxiliary results.

\subsection{Local Harnack inequalities}

\begin{lemma}
There exists a constant \,$C\!>\!0$ such that
\begin{equation*}
p_{\ssf n+1}(x)\ge C\,p_{\ssf n}(y)
\end{equation*}
for every \,$n\!\in\!\N^*$ \ssb and for all neighbors \,$x,y\!\in\!\X$.
\end{lemma}

This inequality is an immediate consequence of the very definition
\begin{equation*}
p_{\ssf n+1}(x)=A\,p_n(x)=\tfrac1{2\ssf(q^{\ssf2}+\ssf q\ssf+1)}
\hspace{-2mm}\sum_{\substack{y\in\X\\d(y,x)=1}}\hspace{-2mm}
p_{\ssf n}(y)
\qquad\forall\;n\!\in\!\N^*\ssb,\;\forall\;x\!\in\!\X.
\end{equation*}
Next result follows by iteration and by using the fact that
the random walk between two neighbors in $\X$ is \textit{aperiodic\/}, meaning that
\begin{equation*}
\operatorname{gcd}\msf\{\ssf n\msb\in\msb\N^*\msf|\,p_{\ssf n}(x,y)\msb>\msb0\ssf\}=\vsb1
\qquad\forall\,x\ne y\ssf.
\end{equation*}

\begin{corollary}\label{Harnack1}
For every \,$m\!\in\!\N^*$\ssb, there exists \,$C\!>\!0$ such that
\begin{equation}\label{Harnack2}
p_{\ssf n+m}(x)\ge C\,p_{\ssf n}(y)
\end{equation}
for every \,$n\!\in\!\N^*$ \ssb and for all \,$x,y\!\in\!\X$
such that \,$0\ssb<\ssb d(x,y)\ssb\le\ssb m$\ssf.
\end{corollary}

\begin{remark}\label{Harnack3}
Inequality \eqref{Harnack2} remains true when \,$x\ssb=\ssb y$\ssf.
This is proved by the same arguments if \,$m\ssb>\!1$ but, if \,$m\ssb=\!1$\ssf,
this follows only a posteriori from Theorem~\ref{TheoremHeatEstimateA2}.
\end{remark}

\subsection{Remarkable formulae}

The function
\begin{equation*}
h=e^{\ssf\lambda_1}\!+e^{-\lambda_1}\!
+e^{\ssf\lambda_2}\!+e^{-\lambda_2}\!
+e^{\ssf\lambda_1-\lambda_2}\!+e^{\ssf\lambda_2-\lambda_1}
\end{equation*}
enjoys the product formula
\begin{equation}\begin{aligned}\label{ProductFormulaA2}
h\ssb+\ssb2\ssf
&=(\ssf e^{\ssf\lambda_1}\!+\ssb1\ssf)\ssf
(\ssf e^{-\lambda_2}\!+\ssb1\ssf)\ssf
(\ssf e^{\ssf\lambda_2-\lambda_1}\!+\ssb1\ssf)\\
&=(\ssf2\cosh\tfrac{\lambda_1}2)\ssf
(\ssf2\cosh\tfrac{\lambda_2}2)\ssf
(\ssf2\cosh\tfrac{\lambda_1-\lambda_2}2)
\end{aligned}\end{equation}
and the differentiation formula
\begin{equation}\label{DifferentiationFormulaA2}
\boldpi(\partial)\,h^{\ssf n+3}
=(n\!+\!3)^2\ssf(n\!+\!2)\,
\bigl[\,h\ssb+\ssb2\,\tfrac{n\ssf+\ssf1}{n\ssf+\ssf3}\ssf\bigr]\,h^{\ssf n}\ssf\Delta
\qquad\forall\;n\ssb\ge\ssb-\ssf2\,.
\end{equation}
These formulae, which are easily checked
(see Appendix \ref{Appendix} for more general results),
play key roles in our analysis of \eqref{HeatFormula1}.

\subsection{Real phase \ssf$\Phi\ssb=\ssb\Phi_\delta$\ssf.}
\label{RealPhase}

Let \ssf$\delta\!\in\!\cl{\sector}$
with \ssf$|\delta|\!=\ssb\delta_1\!+\ssb\delta_2\!<\!1$\ssf,
where \ssf$\delta_j\!=\!\langle\alpha_j,\delta\ssf\rangle\!\ge\!0$
\ssf are the coordinates of \ssf$\delta$
in the basis $\{\lambda_1,\lambda_2\}$\ssf.
In this subsection, we study the function
\vspace{-1mm}
\begin{equation*}
\Phi(z)=\ssf\Phi_\delta(z)=\ssf\log\tfrac{h(z)}6-\langle\ssf\delta,z\ssf\rangle
\qquad\forall\,z\!\in\!\apartment\ssf,
\end{equation*}
which was introduced in the statement of Theorem \ref{TheoremHeatEstimateA2}
and whose dependency on \ssf$\delta$ \ssf will no more be indicated.

\begin{lemma}\label{PropertiesShiftA2}
\begin{description}[labelindent=0pt,labelwidth=5mm,labelsep*=1pt,leftmargin=!]
\item[\rm(a)]
The function \,$\Phi$ is strictly convex,
tends to \ssf$+\infty$ at infinity
and reaches its minimum \ssf$\phi(\delta)\!\in\!(-\ssb\log6\ssf,0\ssf]$
at a single point \ssf$s\!\in\!\cl{\sector}$,
which is the unique solution to the equation
\begin{equation}\label{equationshift}
\tfrac{dh(s)}{h(s)}=\ssf\delta\,.
\end{equation}
\vspace{-5mm}

Moreover,
\vspace{-.5mm}
\begin{equation}\label{dphi}
d\ssf\phi\ssf(\delta)=-\,s\,.
\end{equation}
\item[\rm(b)]
The map
\,$\delta\ssb\longmapsto\ssb s$
\ssf is real analytic and bijective between
\,$\{\ssf\delta\!\in\!\cl{\sector}\ssf|\,|\delta|\!<\!1\ssf\}$
and \,$\cl{\sector}$.
\item[\rm(c)]
We have
\ssf$h(s)\ssb\approx\ssb e^{\ssf s^1\vee s^2}$\ssb,
where \ssf$s^{\ssf j}\!=\ssb\langle\lambda_j,s\ssf\rangle$
denote the coordinates of \ssf$s$ in the basis \ssf$\{\alpha_1,\alpha_2\}$\ssf.
\item[\rm(d)]
The jth coordinates \ssf$\delta_j$
and \ssf$s_j\!=\!\langle\alpha_j,s\ssf\rangle$
in the basis \ssf$\{\lambda_1,\lambda_2\}$
vanish simultaneously.
\item[\rm(e)]
$\,|\delta|\ssb\to\ssb1$
if and only if \ssf$|s|\ssb\to\ssb+\infty$\ssf.
More precisely,
\ssf$1\!-\ssb|\delta|\approx
e^{-\ssf(s^1\ssb\wedge\ssf s^2)}$\ssf.
\item[\rm(f)]
We have \ssf$1\!-\ssb\delta_1\!\vee\ssb\delta_2\ssb
\approx e^{-\ssf|s^1\ssb-\ssf s^2|}$.
More precisely,
\begin{equation*}
e^{-\ssf|s^1\ssb-\ssf s^2|}
=\tfrac{1\ssf-\ssf\delta_1\ssb\vee\delta_2\vphantom{|}}
{\delta_1\ssb\vee\delta_2}
+\O\ssf(\ssf1\!-\ssb|\delta|\ssf)
\quad\text{as \,}|\delta|\ssb\to\ssb1\ssf.
\end{equation*}
\item[\rm(g)]
$\,\delta_1\!-\ssb\delta_2$ and
\,$s_1\!-\ssb s_2\!=\ssb3\ssf(s^1\!-\ssb s^2)$ have the same sign.
Moreover
\begin{equation*}
|\ssf\delta_1\!-\ssb\delta_2\ssf|\approx1\!-\ssb e^{-\ssf|s^1\ssb-\ssf s^2|}\approx1\ssb\wedge\ssb|\ssf s_1\!-\ssb s_2\ssf|\,.
\end{equation*}
\end{description}
\end{lemma}

\begin{proof}
As \ssf$h(z)\ssb>\ssb e^{\ssf|z^1|\vee|z^2|}$ \ssf and
\ssf$|\langle\delta,z\rangle|\ssb
\le\ssb|\delta|\ssf(\ssf|z^1|\ssb\vee\ssb|z^2|\ssf)$\ssf,
where \ssf$z^{\ssf j}\!=\ssb\langle\lambda_j,z\ssf\rangle$
\ssf are the coordinates of \ssf$z$
\ssf in the basis $\{\alpha_1,\alpha_2\}$\ssf,
we have
\begin{equation*}
\Phi(z)>(\ssf1\ssb-\ssb|\delta|\ssf)\,(\ssf|z^1|\ssb\vee\ssb|z^2|\ssf)-\log6\,.
\end{equation*}
Hence \ssf$\Phi(z)\ssb\to\ssb+\infty$
\ssf as \ssf$z$ \ssf tends to infinity in \ssf$\apartment$\ssf.
Moreover, as \ssf$h$ \ssf is $W_0$\ssf-\ssf invariant and
\begin{equation*}
\langle\ssf\delta,z\ssf\rangle\ge\langle\ssf\delta,w.z\ssf\rangle
\qquad\forall\,z\!\in\!\cl{\sector},
\;\forall\,w\!\in\!W_0\ssf,
\end{equation*}
$\Phi$ \ssf reaches its minimum $\phi(\delta)$
in $\cl{\sector}$\ssb.
Notice that $\phi(\delta)\!>\!-\log6$
and $\phi(\delta)\!\le\!\Phi(0)\!=\ssb0$.
Let us next compute the first two derivatives of \ssf$\log h$\ssf.
As \ssf$h\ssb
=\ssb\sum_{\ssf\lambda\in\Lambda}\ssb e^{\ssf\lambda}$
\ssf is a sum of exponentials,
the gradient of \ssf$\log h$ \ssf is given by
\begin{equation}\label{gradient}
d\ssf(\log h)=\ssf\tfrac1h\,
\sum\nolimits_{\ssf\lambda\in\Lambda}\ssb
e^{\ssf\lambda}\ssf\lambda
\end{equation}
and its Hessian by
\begin{equation}\label{Hessian1}\begin{aligned}
d^{\ssf2}(\log h)
&=\ssf\tfrac1h\,
\sum\nolimits_{\ssf\lambda\in\Lambda}\ssb
e^{\ssf\lambda}\,\lambda\!\otimes\!\lambda
-\ssf\tfrac1{h^2}\,
\sum\nolimits_{\ssf\lambda,\lambda'\ssb\in\Lambda}\ssb
e^{\ssf\lambda+\lambda'}\lambda\!\otimes\!\lambda'\\
&=\ssf\tfrac1{2\ssf h^2}\,
\sum\nolimits_{\ssf\lambda,\lambda'\ssb\in\Lambda}\ssb
e^{\ssf\lambda+\lambda'}
(\lambda\!-\!\lambda')\!\otimes\!(\lambda\!-\!\lambda')\,.
\end{aligned}\end{equation}
Since the vectors \ssf$\lambda\!-\!\lambda'$ span \ssf$\apartment$\ssf,
we conclude that the Hessian
\,$d^{\ssf2\ssf}\Phi\ssb=\ssb d^{\ssf2}(\log h)$
\ssf is positive definite,
that \ssf$\Phi$ \ssf is strictly convex
and that \ssf$\Phi$ \ssf has a single minimum~\ssf$s$\ssf,
\linebreak
which satisfies the stationary equation \eqref{equationshift}.
Moreover \,$s\ssb=\ssb s(\delta)$ \ssf depends analytically on \ssf$\delta$,
according to the local inversion theorem,
and the derivative of
\begin{equation*}
\phi\ssf(\delta)=\ssf\Phi\ssf(s\ssf(\delta))
=\ssf\log\tfrac{h(s(\delta))}6-\langle\ssf\delta\ssf,\ssf s(\delta)\ssf\rangle
\end{equation*}
is given by
\begin{equation*}
d\ssf\phi\ssf(\delta)
=d\ssf\Phi\ssf(s(\delta))\circ d\ssf s\ssf(\delta)-s\ssf(\delta)
=-\,s\ssf(\delta)\,.
\end{equation*}
This proves (a) and the first part of (b).
Apart from (c), which is obvious,
all other claims rely on \eqref{equationshift},
which is equivalent to the system
\begin{equation}\label{systemshift}\begin{cases}
\;h(s)\,\delta_1=2\ssf\sinh s^1\!+2\ssf\sinh\ssf(\ssf s^1\!-\ssb s^2\ssf)
=4\ssf\sinh\frac{s_1}2\ssf\cosh\frac{s^2}2\ssf,\\
\;h(s)\,\delta_2=2\ssf\sinh s^2\!+2\ssf\sinh\ssf(\ssf s^2\!-\ssb s^1\ssf)
=4\ssf\sinh\frac{s_2}2\ssf\cosh\frac{s^1}2\ssf.\\
\end{cases}\end{equation}
Firstly, observe that
\begin{equation*}
s_1\left\{\begin{matrix}>\\\,=\,\\<\end{matrix}\right\}\ssf s_2
\quad\Longleftrightarrow\quad
s^1\left\{\begin{matrix}>\\\,=\,\\<\end{matrix}\right\}\ssf s^2
\quad\Longleftrightarrow\quad
\delta_1\left\{\begin{matrix}>\\\,=\,\\<\end{matrix}\right\}\ssf\delta_2
\end{equation*}
and
\begin{equation*}
\delta_j=0
\hspace{2mm}\Longleftrightarrow\hspace{2mm}
s_j=0\,.
\end{equation*}
Secondly, by adding up the equations in \eqref{systemshift}, we get
\begin{equation*}
h(s)\,|\delta|=2\, (\ssf\sinh s^1\!+\sinh s^2\ssf)\ssf.
\end{equation*}
On the one hand, we deduce that
\begin{equation*}
1-|\delta|\,
=\,\tfrac{2\,e^{-s^1}\ssb+\,2\,e^{-s^2}\ssb+\,\cosh\ssf(s^1\ssb-s^2)}
{h(s)\vphantom{0^0}}
\approx\,\tfrac{e^{\ssf|s^1\ssb-s^2|}}{e^{\,s^1\vee s^2}}\,
\approx\,e^{-\ssf s^1\ssb\wedge\ssf s^2}\,,
\end{equation*}
hence
\begin{equation*}
|\delta|\ssb\to\ssb1
\quad\Longleftrightarrow\quad
|s|\ssb\to\ssb+\infty\,.
\end{equation*}
On the other hand,
given \ssf$s\!\in\!\cl{\sector}$,
\eqref{systemshift} defines \ssf$\delta\!\in\!\cl{\sector}$
with
\begin{equation*}
|\delta|=\tfrac{\sinh s^1+\,\sinh s^2\vphantom{|}}
{\cosh s^1+\,\cosh s^2+\,\cosh\ssf(s^1\ssb-s^2)\vphantom{|}}
<1\,.
\end{equation*}
Thirdly, by substracting the equations in \eqref{systemshift}, we get
\begin{align*}
h(s)\,|\ssf\delta_1\!-\ssb\delta_2\ssf|
&=2\,|\sinh s^1\!-\sinh s^2\ssf|+4\,\sinh|\ssf s^1\!-\ssb s^2\ssf|\\
&=\bigl(4\ssf\cosh\tfrac{s^1\ssb+\ssf s^2\vphantom{|}}2
+8\ssf\cosh\tfrac{s^1\ssb-\ssf s^2\vphantom{|}}2\bigr)
\sinh\tfrac{|s^1\ssb-\ssf s^2|}2\,,
\end{align*}
hence
\begin{equation*}
|\ssf\delta_1\!-\ssb\delta_2\ssf|\,
\approx\,\tfrac{e^{\ssf\frac{s^1\ssb+\ssf s^2}2}\ssf
\sinh\frac{|s^1\ssb-\ssf s^2|}2\vphantom{\big{|}}}
{e^{\,s^1\vee s^2}\vphantom{\big{|}}}
\approx1\!-\ssb e^{-\ssf|s^1\ssb-\ssf s^2|}
\approx1\ssb\wedge\ssb|\ssf s^1\!-\ssb s^2\ssf|\,
\approx\,1\ssb\wedge\ssb|\ssf s_1\!-\ssb s_2\ssf|
\end{equation*}
and
\begin{equation*}
\delta_1\!-\ssb\delta_2\to0
\quad\Longleftrightarrow\quad
s_1\!-\ssb s_2\to0\,.
\end{equation*}
Fourthly, by rewriting \eqref{systemshift} as follows\,:
\begin{equation*}\begin{cases}
\;\frac{h(s)}2\,(\ssf1\!-\ssb\delta_1)
=\cosh s^2\ssb+e^{-s^1}\!+e^{\ssf s^2\ssb-\ssf s^1},\\
\;\frac{h(s)}2\,(\ssf1\!-\ssb\delta_2)
=\cosh s^1\ssb+e^{-s^2}\!+e^{\ssf s^1\ssb-\ssf s^2},\\
\end{cases}\end{equation*}
we get on the one hand
\begin{equation*}
1\ssb-\delta_1\!\vee\ssb\delta_2\,\approx\,e^{-|s^1\ssb-\ssf s^2|}
\end{equation*}
and on the other hand
\begin{equation*}
\tfrac{1\ssf-\,\delta_1\ssb\vee\delta_2}{\delta_1\ssb\vee\delta_2}=
e^{-|s^1\ssb-\ssf s^2|}+\O\bigl(e^{-\ssf s^1\ssb\wedge\ssf s^2}\bigr)\ssf.
\end{equation*}
This concludes the proof of Lemma \ref{PropertiesShiftA2}
\end{proof}

By symmetry, we may assume from now that \ssf$x_1\!\ge\ssb x_2$\ssf,
which amounts to either condition \ssf$\delta_1\!\ge\ssb\delta_2$\ssf, $s^1\!\ge\ssb s^2$ or \ssf$s_1\!\ge\ssb s_2$\ssf,
according to Lemma \ref{PropertiesShiftA2}.(g).
Beside the walls
\ssf$\{\alpha_1\!=\ssb0\ssf\}\ssb=\ssb\R\ssf\lambda_1$
and \ssf$\{\alpha_2\!=\ssb0\ssf\}\ssb=\ssb\R\ssf\lambda_2$
of the Weyl chamber $\sector$,
consider the \textit{extra wall\/}
\begin{equation}\label{ExtraWall}
\{\ssf\alpha_1\!=\ssb\alpha_2\}
=\{\ssf\lambda_1\!=\ssb\lambda_2\}
=\R\,\rho\,.
\end{equation}
\vspace{-4mm}

\begin{corollary}\label{NonzeroDenominator}
Assume that \,$\delta$ or equivalently \,$s$ stays away from \,$0$\ssf.
\begin{description}[labelindent=0pt,labelwidth=5mm,labelsep*=1pt,leftmargin=!]
\item[\rm(a)]
The following estimate holds,
provided that \,$x^+$ stays far enough from the extra wall \eqref{ExtraWall}\,{\rm:}
\vspace{-1mm}
\begin{equation*} 
\bigl|\,h(s\!+\!i\ssf\theta)\ssb+\ssb2\,\tfrac{n\ssf+\ssf1}{n\ssf+\ssf3}\ssf\bigr|
\gtrsim\tfrac{h(s)}n\,.
\end{equation*}
\item[\rm(b)]
The improved estimate
\vspace{-.5mm}
\begin{equation*} 
\bigl|\,h(s\!+\!i\ssf\theta)\ssb+\ssb2\,\tfrac{n\ssf+\ssf1}{n\ssf+\ssf3}\ssf\bigr|
\approx h(s)
\end{equation*}
\vspace{-4.5mm}

holds in the following two cases\,{\rm:}
\begin{description}[labelindent=2pt,labelwidth=4mm,labelsep*=1pt,leftmargin=!]
\item[\textbullet]
$\theta$ \ssf is close enough to the extra wall \eqref{ExtraWall}
and \,$n$ is large enough,
\item[\textbullet]
$\delta$ or equivalently \,$s$
stays away from the extra wall \eqref{ExtraWall}
and \,$n$ is large enough.
\end{description}
\end{description}
\end{corollary}

\begin{proof}
The upper estimate
\begin{equation*}
\bigl|\,h(s\!+\!i\ssf\theta)\ssb+\ssb2\,\tfrac{n\ssf+\ssf1}{n\ssf+\ssf3}\ssf\bigr|
\le h(s)\ssb+\ssb2
\le\tfrac43\,h(s)
\end{equation*}
is elementary and holds in full generality.
Let us turn to the lower estimates.
We deduce from \eqref{ProductFormulaA2} that
\begin{equation*}
|\,h(s\!+\!i\ssf\theta)\ssb+\ssb2\,|\ge
\bigl\{e^{\ssf s^1}\hspace{-1mm}-\ssb1\bigr\}\ssf
\bigl\{1\ssb-e^{-s^2}\bigr\}\,
\bigl|\ssf e^{\,i\ssf(\theta^1\ssb-\ssf\theta^2)}\!
+\ssb e^{-(s^1\ssb-\ssf s^2)}\bigr|
\end{equation*}
and we use Lemma \ref{PropertiesShiftA2} 
to estimate the three factors on the right hand side.
On the one hand,
\begin{equation*}
e^{\ssf s^1}\hspace{-1mm}-\ssb1\,\approx\,e^{\ssf s^1}\ssb\approx\,h(s)
\quad\text{and}\quad
1\ssb-e^{-s^2}\ssb\approx\,1\,.
\end{equation*}
On the other hand,
\begin{equation*}
\bigl|\ssf e^{\,i\ssf(\theta^1\ssb-\ssf\theta^2)}\!
+\ssb e^{-(s^1\ssb-\ssf s^2)}\bigr|
\ge1\ssb-e^{-(s^1\ssb-\ssf s^2)}
\approx\ssf\delta_1\!-\ssb\delta_2
=\tfrac{x_1-\ssf x_2}{n\ssf+\ssf2}
\end{equation*}
in general and
\begin{equation*}
\bigl|\ssf e^{\,i\ssf(\theta^1\ssb-\ssf\theta^2)}\!
+\ssb e^{-(s^1\ssb-\ssf s^2)}\bigr|
\approx1
\end{equation*}
if \,$\theta^1\hspace{-1mm}-\!\theta^2$ is small.
Consequently,
\begin{equation*}
\bigl|\,h(s\!+\!i\ssf\theta)\ssb+\ssb2\,\tfrac{n\ssf+\ssf1}{n\ssf+\ssf3}\ssf\bigr|
\ge|\,h(s\!+\!i\ssf\theta)\ssb+\ssb2\,|-\tfrac4{n\ssf+\ssf3}
\end{equation*}
is bounded from below by \ssf$\frac{h(s)}n$,
if \ssf$x_1\!-\ssb x_2$ is large enough,
and by \ssf$h(s)$\ssf,
if \ssf$n$ \ssf is large enough
and if \,$\delta_1\!-\ssb\delta_2$ \ssf stays away from \ssf$0$
\ssf or if \,$\theta^1\hspace{-1mm}-\!\theta^2$ is sufficiently small.
\end{proof} 

\subsection{Complex phase \ssf$\Psi$\ssf.}
\label{ComplexPhase}

In this subsection, we study the function
\begin{equation*}
\Psi(\theta)
=\log\tfrac{h(s\ssf+\ssf i\ssf\theta)}{h(s)}
-i\ssf\langle\ssf\delta,\theta\ssf\rangle\ssf,
\end{equation*}
which does depend on \ssf$\delta$\ssf, or equivalently on \ssf$s$\ssf,
and which is well defined in a neighborhood of the origin in \ssf$\apartment$\ssf,
but independently of \ssf$\delta$ \ssf and \ssf$s$\ssf.
We have \ssf$\Psi(0)\ssb=\ssb0$ \ssf by definition
and \ssf$d\ssf\Psi(0)\ssb=\ssb0$
\ssf according to our choice of \ssf$s$\ssf.
Next lemma describes more precisely
the behavior of \,$\Psi$ \ssf near the origin.

\begin{lemma}\label{LocalBehaviorPsi}
The following results hold uniformly with respect to \,$\delta$ and \,$s$\;{\rm:}
\begin{description}[labelindent=0pt,labelwidth=5mm,labelsep*=1pt,leftmargin=!]
\item[\rm(a)]
The Hessian \,$d^{\ssf2}\Psi(0)$ is negative definite
and \,$B\ssb=\ssb-\,d^{\ssf2}\Psi(0)$ satisfies
\begin{equation*}
B(\theta,\theta)\approx\ssf
e^{-(s^1\ssb-\ssf s^2)}\ssf(\theta^1\!-\ssb\theta^2)^2
+e^{-s^2}(\theta^1\!+\ssb\theta^2)^2
\quad\forall\;\theta\!\in\!\apartment\,,
\end{equation*}
with \,$e^{-(s^1\ssb-\ssf s^2)}\ssb\ge\ssb e^{-s^2}$\ssb.
\item[\rm(b)]
For \,$\theta$ small,
\begin{equation*}
-\Re\Psi(\theta)\approx B(\theta,\theta)
\quad\text{and}\quad
\Im\Psi(\theta)=\O\ssf\bigl(\ssf|\theta|\,B(\theta,\theta)\bigr)\ssf.
\end{equation*}
\end{description}
\end{lemma}

\begin{proof}
Let us compute the Hessian of \ssf$\Psi$,
as we did for \ssf$\log h$ \ssf in \eqref{Hessian1}\,:
\begin{equation}\label{Hessian2}\begin{aligned}
&d^{\ssf2}\Psi(\theta)
=\ssf\tfrac{-1}{2\,h(s+i\ssf\theta)^2}\,
\sum\nolimits_{\ssf\lambda,\lambda'\ssb\in\Lambda}\ssb
e^{\ssf\langle\lambda+\lambda'\!,\ssf s+i\ssf\theta\ssf\rangle}\ssf
(\lambda\!-\!\lambda')\!\otimes\!(\lambda\!-\!\lambda')\\
&=\ssf\tfrac{-1}{2\,|h(s+i\ssf\theta)|^4}\,
\sum\nolimits_{\ssf\lambda,\lambda'\!,\mu,\mu'\ssb\in\Lambda}\ssb
e^{\ssf\langle\lambda+\lambda'\!+\mu+\mu'\!,\ssf s\ssf\rangle}\ssf
e^{\ssf i\ssf\langle\lambda+\lambda'\!-\mu-\mu'\!,\ssf\theta\ssf\rangle}\ssf
(\lambda\!-\!\lambda')\!\otimes\!(\lambda\!-\!\lambda')\,.
\end{aligned}\end{equation}
Observe that, in the nonnegative quadratic form
\begin{equation*}
B(\theta,\theta)=\tfrac1{2\,h(s)^2}\,
\sum\nolimits_{\ssf\lambda,\lambda'\ssb\in\Lambda}\ssb
e^{\ssf\langle\lambda+\lambda'\!,\ssf s\ssf\rangle}\ssf
\langle\ssf\lambda\!-\!\lambda'\ssb,\theta\ssf\rangle^2\,,
\end{equation*}
the leading coefficients arise when
\begin{equation}\label{choices}
\{\lambda,\lambda'\}=\begin{cases}
\ssf\{\lambda_1,\lambda_2\}\ssf,\\
\ssf\{\lambda_1,\lambda_1\!-\!\lambda_2\}\ssf.\\
\end{cases}\end{equation}
Consequently,
\begin{equation}\label{Hessian3}
B(\theta,\theta)\approx\ssf h(s)^{-2}\,\bigl\{\ssf
e^{\ssf\langle\ssf\rho,\ssf s\ssf\rangle}\ssf
\langle\ssf\lambda_1\!-\!\lambda_2,\theta\ssf\rangle^2\ssb
+e^{\ssf\langle\ssf\alpha_1,\ssf s\ssf\rangle}\ssf
\langle\ssf\lambda_2,\theta\ssf\rangle^2\ssf\bigr\}\ssf,
\end{equation}
with \,$h(s)\approx e^{\ssf\langle\ssf\lambda_1,\ssf s\ssf\rangle}$\ssf.
Furthermore, as \ssf$e^{\ssf\langle\ssf\rho,s\ssf\rangle}\!
\ge\ssb e^{\ssf\langle\ssf\alpha_1,s\ssf\rangle}$\ssf,
one may replace in \eqref{Hessian3}
\ssf$\langle\ssf\lambda_2,\theta\ssf\rangle^2$
\ssf by \ssf$\langle\ssf\lambda_1\!+\!\lambda_2,\theta\ssf\rangle^2$\ssf.
This proves (a) and the proof of (b) is somewhat similar.
According to Taylor's formula and \eqref{Hessian2},
we have indeed
\vspace{-1mm}
\begin{equation*}
\Psi(\theta)=\int_{\,0}^{\ssf1}
\langle\ssf d^{\ssf2}\Psi(t\ssf\theta)\ssf\theta,\theta\ssf\rangle\,
(1\!-\ssb t)\,dt\,,
\end{equation*}
\vspace{-4mm}

with
\vspace{1mm}
\begin{equation}\label{Hessian4}
\langle\ssf d^{\ssf2}\Psi(t\ssf\theta)\ssf\theta,\theta\ssf\rangle
=\tfrac{-1}{2\,|h(s+i\ssf t\ssf\theta)|^4}\hspace{-2mm}
\sum_{\lambda,\lambda'\!,\mu,\mu'\ssb\in\Lambda}\hspace{-2mm}
e^{\ssf\langle\lambda+\lambda'\!+\mu+\mu'\!,\ssf s\ssf\rangle}\ssf
e^{\ssf i\ssf\langle\lambda+\lambda'\!-\mu-\mu'\!,\ssf t\ssf \theta\ssf\rangle}\ssf
\langle\ssf\lambda\!-\!\lambda'\ssb,\theta\ssf\rangle^2\ssf.
\end{equation}
\vspace{-2.5mm}

The leading coefficients in \eqref{Hessian4} are obtained
by taking \ssf$\{\lambda,\lambda^{\ssf\prime}\}$ \ssf as in \eqref{choices}
and \ssf$\mu\ssb=\ssb\mu^{\ssf\prime}\!=\ssb\lambda_1$.
Hence, if \ssf$\theta$ \ssf is small,
\begin{align*}
-\Re\,\langle\ssf d^{\ssf2}\Psi(t\ssf\theta)\ssf\theta,\theta\ssf\rangle
&=\ssf|h(s\!+\!i\ssf t\ssf\theta)|^{-4}\\
&\ssf\times\hspace{-2mm}
\sum\limits_{\lambda,\lambda'\!,\mu,\mu'\ssb\in\Lambda}\hspace{-2mm}
e^{\ssf\langle\lambda+\lambda'\!+\mu+\mu'\!,\ssf s\ssf\rangle}\,
\underbrace{\cos\,\langle\ssf\lambda\!+\!\lambda^{\ssf\prime}\hspace{-1mm}
-\!\mu\!-\!\mu^{\ssf\prime}\ssb,\ssf t\,\theta\ssf\rangle}_{\ge\,\const\ssf>\,0}\;
\langle\ssf\lambda\!-\!\lambda'\ssb,\theta\ssf\rangle^2
\end{align*}
\vspace{-3mm}

is comparable to \eqref{Hessian3}, while
\begin{align*}
-\Im\,\langle\ssf d^{\ssf2}\Psi(t\ssf\theta)\ssf\theta,\theta\ssf\rangle
&=\ssf|h(s\!+\!i\ssf t\ssf\theta)|^{-4}\\
&\ssf\times\hspace{-2mm}
\sum\limits_{\lambda,\lambda'\!,\mu,\mu'\ssb\in\Lambda}\hspace{-2mm}
e^{\ssf\langle\lambda+\lambda'\!+\mu+\mu'\!,\ssf s\ssf\rangle}\,
\underbrace{\sin\,\langle\ssf\lambda\!+\!\lambda^{\ssf\prime}\hspace{-1mm}
-\!\mu\!-\!\mu^{\ssf\prime}\ssb,\ssf t\,\theta\ssf\rangle}_{\O(|\theta|)}\;
\langle\ssf\lambda\!-\!\lambda'\ssb,\theta\ssf\rangle^2
\end{align*}
\vspace{-3.5mm}

is \ssf$\O\ssf\bigl(\ssf|\theta|\,B(\theta,\theta)\bigr)$\ssf.
\end{proof}

Beside the local behavior of $\Psi$,
we shall need the following global estimate.

\begin{lemma}\label{GlobalEstimatePsi}
For every \,$\theta\!\in\!U$\ssb,
\begin{equation*}
-\ssf\log\tfrac{|h(s\ssf+\ssf i\ssf\theta)|}{h(s)}
\ssf\gtrsim B(\theta,\theta)\ssf.
\end{equation*}
\end{lemma}

\begin{remark}
Such a global estimate is hard to obtain for general random walks on affine buildings $($see \cite{T}$)$.
\end{remark}

\begin{figure}[h]
\begin{center}
\psfrag{0}[c]{$0$}
\psfrag{bullet}[c]{\textbullet}
\psfrag{U}[c]{$U$}
\psfrag{pialpha1}[l]{$\pi\ssf\alpha_1$}
\psfrag{pialpha2}[r]{$\pi\ssf\alpha_2$}
\psfrag{pirho}[c]{$\pi\ssf\rho$}
\psfrag{2pilambda1}[l]{$2\ssf\pi\lambda_1$}
\psfrag{2pilambda2}[r]{$2\ssf\pi\lambda_2$}
\psfrag{2pilambda1lambda2}[l]{$2\ssf\pi\ssf(\lambda_1\!-\ssb\lambda_2)$}
\includegraphics[width=100mm]{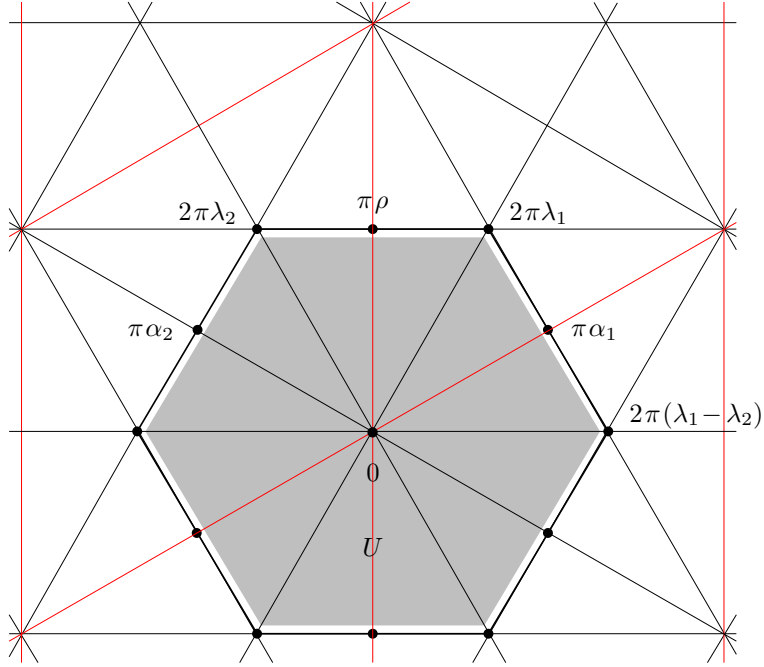}
\end{center}
\caption{Picture for the proof of Lemma \ref{GlobalEstimatePsi}}
\label{zeroes}
\end{figure}

\begin{proof}[Proof of Lemma \ref{GlobalEstimatePsi}]
Let us expand
\begin{equation*}\begin{aligned}
h(s)^2\!-\ssb|h(s\!+\!i\ssf\theta)|^2
&=\sum\nolimits_{\ssf\lambda,\lambda'\ssb\in\Lambda}\ssb
e^{\ssf\langle\lambda+\lambda'\!,\ssf s\ssf\rangle}\,\bigl\{\ssf1\!-
\ssb\cos\,\langle\lambda\!-\!\lambda^{\ssf\prime}\ssb,\theta\ssf\rangle\bigr\}\\
&=\ssf2\ssf\sum\nolimits_{\ssf\lambda,\lambda'\ssb\in\Lambda}\ssb
\cosh\ssf\langle\lambda\!+\!\lambda'\ssb,s\ssf\rangle\ssf
\sin^2\ssb\tfrac{\langle\lambda-\lambda^{\ssf\prime}\!,\ssf\theta\ssf\rangle}2\,.
\end{aligned}\end{equation*}
\vspace{-3.5mm}

By taking
\vspace{-0.5mm}
\begin{equation*}
\{\lambda,\lambda'\}=\begin{cases}
\ssf\pm\ssf\{\lambda_1,\lambda_2\}\ssf,\\
\ssf\pm\ssf\{\lambda_1,\lambda_1\!-\!\lambda_2\}\ssf,\\
\end{cases}\end{equation*}
we get the lower bound
\begin{equation*}
h(s)^2\ssb-|\ssf h(s\!+\!i\ssf\theta)|^2
\ge\ssf8\ssf\cosh\ssf\langle\ssf\rho,s\ssf\rangle\ssf
\sin^2\ssb\tfrac{\langle\lambda_1\ssb-\lambda_2,\ssf\theta\ssf\rangle}2
+8\ssf\cosh\ssf\langle\alpha_1,s\ssf\rangle\ssf
\sin^2\ssb\tfrac{\langle\lambda_2,\ssf\theta\ssf\rangle}2\,.
\end{equation*}
As \ssf$\|\lambda_1\!-\!\lambda_2\|^2\!=\ssb\frac23$
\ssf and \ssf$\|\lambda_2\|^2\!=\ssb\frac23$\ssf,
we have
\begin{equation*}
\bigl|\tfrac{\langle\lambda_1-\lambda_2,\ssf\theta\ssf\rangle}2\bigr|\ssb
\le\ssb\tfrac{2\ssf\pi}3
\quad\text{and}\quad
\bigl|\tfrac{\langle\lambda_2,\ssf\theta\ssf\rangle}2\bigr|\ssb
\le\ssb\tfrac{2\ssf\pi}3
\end{equation*}
on \ssf$U$ (see Figure \ref{zeroes}),
hence
\begin{equation*}
\sin^2\ssb\tfrac{\langle\lambda_1\ssb-\lambda_2,\ssf\theta\ssf\rangle}2
\approx\langle\lambda_1\hspace{-.75mm}-\!\lambda_2,\theta\ssf\rangle^2
\quad\text{and}\quad
\sin^2\ssb\tfrac{\langle\lambda_2,\ssf\theta\ssf\rangle}2
\approx\langle\lambda_2,\theta\ssf\rangle^2\ssf.
\end{equation*}
By using \,$h(s)\approx e^{\ssf\langle\ssf\lambda_1,\ssf s\ssf\rangle}$\ssf,
we deduce that
\begin{equation*}
\tfrac{h(s)^2-\ssf|h(s\ssf+\ssf i\ssf\theta)|^2}{2\,h(s)^2}
\gtrsim
e^{-\langle\ssf\lambda_1\ssb-\lambda_2,\ssf s\ssf\rangle}\ssf
\langle\ssf\lambda_1\!-\!\lambda_2,\theta\ssf\rangle^2\ssb
+e^{-\langle\ssf\lambda_2,\ssf s\ssf\rangle}\ssf
\langle\ssf\lambda_2,\theta\ssf\rangle^2\ssf,
\end{equation*}
where we may again replace \ssf$\langle\ssf\lambda_2,\theta\ssf\rangle^2$
\ssf by \ssf$\langle\ssf\lambda_1\!+\ssb\lambda_2,\theta\ssf\rangle^2$\ssf.
We conclude by using Lemma \ref{LocalBehaviorPsi}.(a) and the elementary estimate
\begin{equation*}
-\log\tfrac{|h(s\ssf+\ssf i\ssf\theta)|}{h(s)}
=-\,\tfrac12\log\ssf\bigl\{\ssf1\ssb
-\tfrac{h(s)^2-\ssf|h(s\ssf+\ssf i\ssf\theta)|^2}{h(s)^2}\bigr\}
\ge\tfrac{h(s)^2-\ssf|h(s\ssf+\ssf i\ssf\theta)|^2}{2\,h(s)^2}\,.
\end{equation*}
\end{proof}

\subsection{Amplitudes}
\label{Amplitudes}

In this subsection, we study
the following amplitudes occurring in \eqref{Integrals}\,:
\vspace{-2mm}
\begin{align}
a_1(\theta)&=\smash{
|\ssf\mathbf{c}(i\ssf\theta)|^{-2}\,,
}\vphantom{\Big|}\label{Amplitude1}\\
a_2(\theta)&=\smash{
\tfrac{e^{\ssf i\langle x^+\ssb+\rho,\ssf\theta\ssf\rangle}
\vphantom{\frac00}}
{\boldpi(x^+\ssb+\ssf\rho)
\vphantom{\left[\frac00\right]}}\,
\boldpi\bigl(i\tfrac\partial{\partial\theta}\bigr)\ssf
\tfrac{e^{-i\langle x^+\ssb+\rho,\ssf\theta\ssf\rangle}
\vphantom{\frac00}}
{\mathbf{b}(s\ssf+\ssf i\ssf\theta)
\vphantom{\left[\frac00\right]}}\;,
}\vphantom{\Big|}\label{Amplitude2}\\
a_3(\theta)&=\smash{
\tfrac{h(s\ssf+\ssf i\ssf\theta)\,
e^{\ssf i\langle x^+\ssb+\rho,\ssf\theta\ssf\rangle}
\vphantom{\frac00}}
{\boldpi(x^+\ssb+\ssf\rho)
\vphantom{\left[\frac00\right]}}\,
\boldpi\bigl(i\tfrac\partial{\partial\theta}\bigr)\ssf
\tfrac{e^{-i\langle x^+\ssb+\rho,\ssf\theta\ssf\rangle}
\vphantom{\frac00}}
{\mathbf{b}(s\ssf+\ssf i\ssf\theta)\,
\left[\ssf h(s\ssf+\ssf i\ssf\theta)\ssf
+\ssf2\ssf\frac{n+1}{n+3}\ssf\right]}\;.
}\vphantom{\Big|}\label{Amplitude3}
\end{align}

\begin{lemma}\label{LemmaAmplitudes}
\begin{description}[labelindent=0pt,labelwidth=5mm,labelsep*=1pt,leftmargin=!]
\item[\rm(a)]
The function \eqref{Amplitude1} has the following behavior:
\vspace{-1mm}
\begin{equation*}
a_1(\theta)=\boldpi(\theta)^2\,
\Bigl\{\ssf\bigl(\tfrac q{q-1}\bigr)^6\ssb+\ssf\O\ssf(|\theta|)\ssf\Bigr\}\,.
\end{equation*}
\vspace{-4mm}

\item[\rm(b)]
The function \eqref{Amplitude2} is uniformly bounded,
as well as its derivatives.
Moreover \,$|\ssf a_2(0)|\approx1$
\ssb provided that \,$x$ is large enough.

\item[\rm(c)]
The function \eqref{Amplitude3} is \,$\O\ssf(n^4)$\ssf,
provided that \,$x_1\!-\ssb x_2$ is large enough.
It is actually bounded, as well as its derivatives,
in the following two cases\,{\rm:}
\begin{description}[labelindent=2pt,labelwidth=4mm,labelsep*=1pt,leftmargin=!]
\item[\textbullet]
$|\ssf\theta_1\!-\ssb\theta_2|$ is small enough
and \,$n$ is large enough,
\item[\textbullet]
$\ssf\delta_1\!-\delta_2$ or equivalently \,$s_1\!-\ssb s_2$ stays away from \,$0$
and \,$n$ is large enough.
\end{description}
Moreover \,$|\ssf a_3(0)|\approx1$
provided that \,$x$ and \,$s$ are large enough.
\end{description}
\end{lemma}

\begin{proof}
(a) is elementary, as well as the first claim in (b).
The first claims in (c) are deduced similarly from Corollary \ref{NonzeroDenominator}.
Let us turn to the lower estimate
\begin{equation}\label{LowerEstimateB}
|\ssf a_2(0)|\gtrsim1
\end{equation}
in (b) and let us expand for this purpose
the expression \eqref{Amplitude2} at \ssf$\theta\ssb=\ssb0$\ssf.
The main term \,$\mathbf{b}(s)^{-1}\!\approx\!1$
\ssf is obtained by applying
\,$\smash{\boldpi\bigl(i\tfrac\partial{\partial\theta}\bigr)
\big|_{\ssf\theta=0}}$
to \,$\smash{e^{-i\langle x^+\ssb+\rho,\ssf\theta\ssf\rangle}}$.
All other terms are \,$\O\bigl(\frac1{x_1+1}\bigr)$\ssf,
except for
\vspace{-2mm}
\begin{equation*}
\tfrac{1\vphantom{\frac{+0}{+0}}}{x_2+1\vphantom{\frac{+0}{+0}}}\,
\tfrac{\partial_{\hspace{.1mm}\alpha_2}\mathbf{b}(s)\vphantom{\frac{+0}{+0}}}
{\mathbf{b}(s)^2\vphantom{\frac{+0}{+0}}}\,.
\end{equation*}
\vspace{-3mm}

This term, which is obtained by differentiating
\,$\mathbf{b}(s\ssb+\ssb i\ssf\theta)^{-1}$
in the direction of \ssf$\alpha_2$
and \,$e^{-i\langle x^+\ssb+\rho,\ssf\theta\ssf\rangle}$
\ssf in the directions of \ssf$\alpha_1$ and \ssf$\rho$\ssf,
happens to be positive, as
\begin{align*}
\tfrac{\partial_{\alpha_2}\mathbf{b}(s)}{\mathbf{b}(s)\vphantom{\frac00}}
&=-\,\tfrac{q^{-1}e^{-s_1}\vphantom{|}}{1-\ssf q^{-1}e^{-s_1}\vphantom{\frac00}}
+\tfrac{2\,q^{-1}e^{-s_2}\vphantom{|}}{1-\ssf q^{-1}e^{-s_2}\vphantom{\frac00}}
+\tfrac{q^{-1}e^{-s_1\ssb-s_2}\vphantom{|}}
{1-\ssf q^{-1}e^{-s_1\ssb-s_2}\vphantom{\frac00}}\\
&>\tfrac{q^{-1}e^{-s_2}\vphantom{|}}{1-\ssf q^{-1}e^{-s_2}\vphantom{\frac00}}
-\tfrac{q^{-1}e^{-s_1}\vphantom{|}}{1-\ssf q^{-1}e^{-s_1}\vphantom{\frac00}}
\ge0\,.
\end{align*}
Hence
\begin{equation*}
|\ssf a_2(0)|
\ge\tfrac{1\vphantom{\frac{+0}{+0}}}{\mathbf{b}(s)\vphantom{\frac{+0}{+0}}}
+\tfrac{1\vphantom{\frac{+0}{+0}}}{x_2+1\vphantom{\frac{+0}{+0}}}\,
\tfrac{\partial_{\hspace{.1mm}\alpha_2}\mathbf{b}(s)\vphantom{\frac{+0}{+0}}}
{\mathbf{b}(s)^2\vphantom{\frac{+0}{+0}}}
-|\ssf\text{remainder}\ssf|
\ge\tfrac{1\vphantom{\frac{+0}{+0}}}{\mathbf{b}(s)\vphantom{\frac{+0}{+0}}}
-\O\bigl(\tfrac{1\vphantom{\frac00}}{x_1+\ssf1}\bigr)
\end{equation*}
is \,$\ge\ssb\smash{\frac1{2\,\mathbf{b}(s)}}\ssb\approx\ssb1$\ssf,
provided that \ssf$x$ \ssf is large enough.
This concludes the proof of \eqref{LowerEstimateB}.
The lower estimate
\begin{equation}\label{LowerEstimateC}
|\ssf a_3(0)|\gtrsim1
\end{equation}
in (c) is proved similarly.
In the expansion of the expression \eqref{Amplitude3} at \ssf$\theta\ssb=\ssb0$\ssf,
the main term is now
\begin{equation*}
T_1=\tfrac{h(s)\vphantom{\frac{+0}{+0}}}
{\mathbf{b}(s)\,\left[\ssf h(s)\vsf+\ssf2\ssf\frac{n+1}{n+3}\ssf\right]}
\ssf\approx1
\end{equation*}
and all other terms are \,$\O\bigl(\frac1{x_1+1}\bigr)$\ssf,
except for
\begin{equation}\label{BadTermC}
\tfrac{1\vphantom{\frac{+0}{+0}}}{x_2+1\vphantom{\frac{+0}{+0}}}\,
\tfrac{h(s)\vphantom{\frac{+0}{+0}}}
{\mathbf{b}(s)\,\left[\ssf h(s)\vsf+\ssf2\ssf\frac{n+1}{n+3}\ssf\right]}\,
\Bigl\{\tfrac{\partial_{\hspace{.1mm}\alpha_2}\mathbf{b}(s)\vphantom{\frac{+0}{+0}}}
{\mathbf{b}(s)\vphantom{\frac{+0}{+0}}}
+\tfrac{\partial_{\hspace{.1mm}\alpha_2}h(s)\vphantom{\frac{+0}{+0}}}
{h(s)\vsf+\ssf2\ssf\frac{n+1}{n+3}\vphantom{\frac{+0}{+0}}}\Bigr\}\,.
\end{equation}
As
\begin{equation*}
\tfrac{\partial_{\hspace{.1mm}\alpha_2}\left[\ssf h(s)\vsf+\ssf2\ssf\right]\vphantom{\frac00}}
{h(s)\ssf+\ssf2\vphantom{\frac{+0}{+0}}}
=-\tfrac{e^{-s^2}\vphantom{\frac00}}{e^{-s^2}+1}
+\tfrac{e^{s^2-s^1}\vphantom{\frac00}}{e^{s^2-s^1}+1}\,,
\end{equation*}
\eqref{BadTermC} is the difference of the positive expressions
\vspace{1mm}
\begin{equation*}\smash{
T_2=\ssf\tfrac{1\vphantom{\frac{+0}{+0}}}{x_2+1\vphantom{\frac{+0}{+0}}}\,
\tfrac{h(s)\vphantom{\frac{+0}{+0}}}
{\mathbf{b}(s)\,\left[\ssf h(s)\vsf+\ssf2\ssf\frac{n+1}{n+3}\ssf\right]}\,\Bigl\{
\tfrac{\partial_{\hspace{.1mm}\alpha_2}\mathbf{b}(s)\vphantom{\frac{+0}{+0}}}
{\mathbf{b}(s)\vphantom{\frac{+0}{+0}}}
+\tfrac{h(s)\vsf+\ssf2\vphantom{\frac{+0}{+0}}}
{h(s)\vsf+\ssf2\ssf\frac{n+1}{n+3}\vphantom{\left[\frac{+0}{+0}\right]}}\,
\tfrac{e^{s^2-s^1}\vphantom{\frac00}}{e^{s^2-s^1}+1}\Bigr\}
}\end{equation*}
and
\begin{equation*}\smash{
T_3=\tfrac{1\vphantom{\frac{+0}{+0}}}{x_2+1\vphantom{\frac{+0}{+0}}}\,
\tfrac{h(s)\,\left[\ssf h(s)\vsf+\ssf2\ssf\right]\vphantom{\frac{+0}{+0}}}
{\mathbf{b}(s)\,\left[\ssf h(s)\vsf+\ssf2\ssf\frac{n+1}{n+3}\ssf\right]
{\vphantom{\frac00}}^2}\,
\tfrac{e^{-s^2}\vphantom{\frac{+0}{+0}}}{e^{-s^2}+1}
=\ssf\O\bigl(e^{-s^2}\ssf\bigr)\ssf.
}\end{equation*}
Hence
\begin{equation*}
|\ssf a_3(0)|
\ge T_1\hspace{-.4mm}+ T_2\ssb-T_3\ssb-|\ssf\text{remainder}\ssf|
\ge T_1\hspace{-.4mm}-\O\bigl(e^{-s^2}\ssf\bigr)\ssb
-\O\bigl(\tfrac{1\vphantom{\frac00}}{x_1+\ssf1}\bigr)
\end{equation*}
is \,$\ge\ssb\frac{T_1}2\ssb\approx\ssb1$\ssf,
provided that \ssf$x$ \ssf and \ssf$s$ \ssf are large enough.
This concludes the proof of \eqref{LowerEstimateC}.
\end{proof}

\subsection{Proof of Theorem \ref{TheoremHeatEstimateA2}
when \ssf$n$ \ssf remains bounded}
\label{Proof1}

Then \eqref{OptimalBoundA2} reduces to
\begin{equation}\label{HeatEstimateNBounded}
p_{\ssf n}(x)\ssf\approx\ssf1
\qquad\forall\hspace{1.25mm}|x|\ssb\le\ssb n\,.
\end{equation}
While the upper bound in \eqref{HeatEstimateNBounded} is trivial,
the lower bound amounts to the nonvanishing of \ssf$p_{\ssf n}(x)$\ssf.
This follows in turn from the fact, already used in the proof of Corollary \ref{Harnack1},
that the random walk is aperiodic and hence that the random walk may join in \ssf$n$ \ssf steps
any two points at distance $\le n$\ssf,
as soon as \ssf$n\ssb\ge\ssb2$\ssf.
\medskip

\subsection{Proof of Theorem \ref{TheoremHeatEstimateA2}
when \ssf $x$ \ssf remains bounded while \ssf$n$ \ssf is large}
\label{Proof2}

Then \eqref{OptimalBoundA2} amounts to
\begin{equation}\label{HeatEstimateXBounded}
p_{\ssf n}(x)\ssf\approx\,n^{-4}\;\boldsigma^{\vsf n}
\end{equation}
as, in this case, the Gaussian type factor
\begin{equation*}
e^{\ssf n\ssf\phi(\delta)}\approx\ssf e^{\ssf(n+2)\ssf\phi(\delta)}
\end{equation*}
is bounded both from above and from below.
The latter claim follows indeed from the mean value theorem,
applied to the function \ssf$\phi$\ssf,
from the vanishing \ssf$\phi\ssf(0)\!=\ssb0$
\ssf and from the boundedness of \ssf$d\ssf\phi$\ssf,
according to \eqref{dphi}.

Although \eqref{HeatEstimateXBounded} is a consequence
of the general local limit theorem in \cite{P3},
we include a short proof,
which will be refined in the next two subsections.
First of all, according to Corollary \ref{Harnack1},
we can reduce to \ssf$x\ssb=\ssb0$\ssf.
Next, by setting \ssf$x\ssb=\ssb0$ \ssf in \eqref{HeatFormula1},
we see that \eqref{HeatEstimateXBounded} amounts to the estimate
\,$J(n)\ssb\approx\ssb n^{-4}$ \ssf
for the nonnegative expression
\begin{equation}\label{DefinitionJn}
J(n)\ssf=\int_{\ssf U}\,
\bigl[\ssf\tfrac{h(i\ssf\theta)}6\ssf\bigr]^n\,a_1(\theta)\;d\ssf\theta\,.
\end{equation}
Let us collect some information about \eqref{DefinitionJn}.
On the one hand, for \ssf$\theta$ \ssf small,
the phase function \ssf$-\ssf\Psi(\theta)\ssb=\ssb-\log\frac{h(i\ssf\theta)}6$
\ssf and the amplitude \ssf$a_1(\theta)$ \ssf
are nonnegative and behave as follows,
according to Lemma \ref{LocalBehaviorPsi}.(b)
and Lemma \ref{LemmaAmplitudes}.(a)\,:
\begin{equation}\label{LocalBehaviorsJn}
-\ssf\Psi(\theta)\approx|\theta|^2
\quad\text{and}\quad
a_1(\theta)\approx\ssf\boldpi(\theta)^2\,.
\end{equation}
On the other hand, for \ssf$\theta\!\in\!U$,
the following global estimates hold,
according to Lemma \ref{GlobalEstimatePsi}
and Lemma \ref{LemmaAmplitudes}.(a)\,:
\begin{equation}\label{GlobalEstimatesJn}
\tfrac{|h(i\ssf\theta)|}6\lesssim\ssf e^{\ssf-\const|\theta|^2}
\qquad\text{and}\qquad
a_1(\theta)\lesssim\boldpi(\theta)^2\,.
\end{equation}
The upper bound of \eqref{DefinitionJn}
is easily deduced from \eqref{GlobalEstimatesJn}\,:
\begin{equation*}
J(n)\ssf\lesssim
\int_{\ssf U}e^{\ssf-\const n\,|\theta|^2}\,\boldpi(\theta)^2\,d\ssf\theta\,
\lesssim\,n^{-4}\,.
\end{equation*}
In order to prove the lower bound,
let us split up
\begin{equation*}
\int_{\,U}\,
=\,\int_{\,\epsilon\ssf U}
+\,\int_{\ssf U\ssf\smallsetminus\,\epsilon\ssf U}
\end{equation*}
and
\begin{equation*}
J(n)=J_1(n)+J_2(n)
\end{equation*}
accordingly,
where \ssf$\epsilon\!\in\!(0,1)$ is chosen small enough,
so that \eqref{LocalBehaviorsJn} holds for \ssf$\theta\!\in\!\epsilon\ssf U$.
Then it follows from \eqref{LocalBehaviorsJn} and \eqref{GlobalEstimatesJn} that
\begin{equation*}
J_1(n)\approx\ssf n^{-4}
\qquad\text{while}\qquad
J_2(n)\lesssim\ssf e^{\ssf-\const n}\,.
\end{equation*}
In conclusion, \ssf$J(n)\ssb\approx n^{-4}$ provided that \ssf$n$ \ssf is large enough.
\medskip

Assume from now on that \ssf$x$ \ssf and hence \ssf$n$ \ssf are large.

\subsection{Proof of Theorem \ref{TheoremHeatEstimateA2}
when \ssf$\frac{|x|}n$ \ssf stays away from $1\ssf$}
\label{Proof3}

Assume that \ssf$|\delta|\ssb\le\ssb1\!-\ssb\eta$\ssf,
with \ssf$\eta\!\in\!(0,1)$ small.
Then the stationary point \ssf$s$ \ssf considered Subsection \ref{RealPhase}
remains bounded, according to Lemma \ref{PropertiesShiftA2}.e.

To begin with, let us modify the integral expression \eqref{HeatFormula1}.
Firstly, by using \eqref{cFunction} and \eqref{MacdonaldPolynomial}, together
with the $W_0$\ssf-\ssf invariance of \ssf$h$ and \ssf$U$\ssb, we get
\begin{equation*}
p_{\ssf n}(x)
=\tfrac1{4\ssf\pi^2}\,\sigma^{\ssf n}\,q^{-\langle\ssf\rho,\ssf x^+\rangle}\ssb
\int_{\ssf U}\ssb h(i\ssf\theta)^n\,\tfrac
{\Delta(i\ssf\theta)\,e^{-i\langle x^+\ssb+\rho,\ssf\theta\ssf\rangle}}
{\mathbf{b}(i\ssf\theta)}\,d\ssf\theta\,.
\end{equation*}
Secondly, by deforming the contour of integration in \ssf$\aC$ and by using the
\ssf$2\ssf\pi\hspace{.1mm}Q$\hspace{.4mm}-\ssf pe\-ri\-od\-ic\-i\-ty
in \ssf$\theta$, we get
\begin{equation*}
p_{\ssf n}(x)
=\tfrac1{4\ssf\pi^2}\,\sigma^{\ssf n}\,q^{-\langle\ssf\rho,\ssf x^+\rangle}\ssf
e^{-\langle\ssf x^+\ssb+\ssf\rho,\ssf s\ssf\rangle}\ssb
\int_{\ssf U}\ssb h(s\ssb+\ssb i\ssf\theta)^n\,\tfrac
{\Delta(s\ssf+\ssf i\ssf\theta)\,e^{-i\langle x^+\ssb+\rho,\ssf\theta\ssf\rangle}}
{\mathbf{b}(s\ssf+\ssf i\ssf\theta)}\,d\ssf\theta\,.
\end{equation*}
Thirdly, after performing an integration by parts
based on \eqref{DifferentiationFormulaA2}, we get
\begin{equation}\label{HeatFormula2}
p_{\ssf n}(x)=\ssf C(n,x^+\ssb)\,J(n,x^+\ssb)\,,
\end{equation}
\vspace{-5.5mm}

where
\vspace{-.5mm}
\begin{equation}\label{DefinitionCnx}
C(n,x^+\ssb)=\tfrac1{4\ssf\pi^2}\ssf\tfrac{\boldpi(x^+\ssb+\ssf\rho)\,}{(n\ssf+\ssf3)^2\ssf(n\ssf+\ssf2)}\,
\sigma^n\ssf q^{-\langle\ssf\rho,\ssf x^+\rangle}\ssf h(s)^{n+2}\ssf e^{-\langle\ssf x^+\ssb+\ssf\rho,\ssf s\ssf\rangle}
\end{equation}
\vspace{-5mm}

and
\vspace{-.5mm}
\begin{equation}\label{DefinitionJnx}
J(n,x^+\ssb)\,=\int_{\ssf U}\,
\bigl[\,\tfrac{h(s\ssf+\ssf i\ssf\theta)}{h(s)}\,
e^{-\ssf i\ssf\langle\ssf\delta,\ssf\theta\ssf\rangle}\ssf\bigr]^{n+2}\,
a_3(\theta)\,d\ssf\theta\,.
\end{equation}

Let us make two observations about the latter expressions.
On the one hand, as
\begin{equation*}\begin{cases}
\;\boldpi(x^+\hspace{-1mm}+\!\rho)
=(x_1\!+\!1)(x_2\!+\!1)(\ssf|x|\!+\!1)\ssf,\\
\;\boldsigma^{\vsf n}\ssb=6^{\hspace{.1mm}n}\,\sigma^{\ssf n}\ssf,\\
\;e^{\ssf n\ssf\phi(\delta)}\ssb
\approx e^{\ssf(n+2)\ssf\phi(\delta)}\ssb
=e^{\ssf(n+2)\ssf\Phi(s)}\ssb
=6^{-n-2}\,h(s)^{n+2}\,
e^{-\langle\ssf x^+\ssb+\ssf\rho,\ssf s\ssf\rangle}\ssf,
\end{cases}\end{equation*}
we have
\begin{equation}\label{EstimateC}
C(n,x^+\ssb)\approx\ssf
\tfrac{(1\ssf+\ssf|x|\ssf)\ssf(1\ssf+\ssf x_1)\ssf(1\ssf+\ssf x_2)}{n^3}\;
\boldsigma^{\vsf n}\,q^{-\langle\ssf\rho,\ssf x^+\rangle}\ssf
e^{\ssf n\ssf\phi(\delta)}\,.
\end{equation}
Hence \eqref{OptimalBoundA2} amounts to
\begin{equation}\label{FirstEstimateJnx}
J(n,x^+\ssb)\ssf\approx\ssf\tfrac1n
\end{equation}
under the current assumptions.
On the other hand,
\eqref{DefinitionJnx} is meaningful
as long as the denominator
\ssf$h(s\ssb+\ssb i\ssf\theta)\ssb+\ssb2\ssf\smash{\frac{n+1}{n+3}}$
in \eqref{Amplitude3} doesn't vanish,
which may happen when \ssf$x^+$ gets close to the extra wall \eqref{ExtraWall},
according to Corollary \ref{NonzeroDenominator}.(a).
We get around this problem by considering
\vspace{-.5mm}
\begin{equation*}
\widetilde{p}_{\ssf n}(x)=\ssf\widetilde{C}(n,x^+\ssb)\,\widetilde{J}(n,x^+\ssb)
\end{equation*}
\vspace{-5.5mm}

instead of \eqref{HeatFormula2}, where
\begin{align}
\widetilde{p}_{\ssf n}(x)
&=p_{\ssf n}(x)+2\,\tfrac n{n\ssf+\ssf2}\,\sigma\,p_{\ssf n-1}(x)\,,
\label{Definitionptildenx}\\
\widetilde{C}(n\vsf,x^+\ssb)
&=\tfrac{(n\ssf+\ssf3)^2}{(n\ssf+\ssf2)\ssf(n\ssf+\vsf1)}\,C(n\vsf,x^+\ssb)
\approx C(n\vsf,x^+\ssb)\,,
\nonumber\\
\widetilde{J}(n,x^+\ssb)&=\!\int_{\ssf U}\,
\bigl[\,\tfrac{h(s\ssf+\ssf i\ssf\theta)}{h(s)}\,
e^{-\ssf i\ssf\langle\ssf\delta,\ssf\theta\ssf\rangle}\ssf\bigr]^{n+2}\,
a_2(\theta)\,d\ssf\theta\,.
\label{DefinitionJtildenx}
\end{align}
Notice indeed that
\begin{equation}\label{ComparisonHeatKernels}
p_{\ssf n}(x)\le\ssf\widetilde{p}_{\ssf n}(x)\lesssim\ssf p_{\ssf n+2}(x)\ssf.
\end{equation}
Here the first inequality is elementary
while the second one follows from the local Harnack inequality
(see Remark \ref{Harnack3}).
\smallskip

Let us prove the estimate \eqref{FirstEstimateJnx}
for the integral \eqref{DefinitionJtildenx}
and, to this end, let us resume
the analysis carried out for \eqref{DefinitionJn}
in Subsection \ref{Proof2}.
The upper bound of \eqref{DefinitionJtildenx} follows again easily
from Lemma \ref{GlobalEstimatePsi} and Lemma \ref{LemmaAmplitudes}.(b).
More precisely,
the Gaussian estimate
\begin{equation}\label{Gaussian}
\tfrac{|h(s\ssf+\ssf i\ssf\theta)|}{h(s)}\ssf
\lesssim\ssf e^{\ssf-\const B(\theta,\theta)}
\le\ssf e^{\ssf-\const|\theta|^2}
\qquad\forall\;\theta\!\in\!U
\end{equation}
and the uniform boundedness of \ssf$a_2$ yield
\begin{equation*}
\widetilde{J}(n,x^+\ssb)\,\lesssim
\int_{\ssf U}\ssb e^{\ssf-\const n\,|\theta|^2}\,d\ssf\theta\,
\lesssim\,\tfrac1n\,.
\end{equation*}
In order to prove the lower bound,
let us split up this time
\begin{equation}\label{DecompositionJtilde}
\widetilde{J}(n,x^+\ssb)\ssf
=\ssf\sum\nolimits_{\ssf k=1}^{\,4}\widetilde{J}_k(n,x^+\ssb)\,,
\end{equation}
\vspace{-4.5mm}

where
\begin{equation*}\begin{aligned}
\widetilde{J}_1(n,x^+\ssb)\ssf
&=\,a_2(0)\int_{\epsilon\ssf U}\ssb
e^{\ssf(n+2)\Re\Psi(\theta)}\,d\ssf\theta\,,\\
\widetilde{J}_2(n,x^+\ssb)\ssf
&=\,a_2(0)\int_{\epsilon\ssf U}
\bigl\{\ssf e^{\ssf(n+2)\ssf\Psi(\theta)}\!
-\ssb e^{\ssf(n+2)\Re\Psi(\theta)}\bigr\}\,d\ssf\theta\,,\\
\widetilde{J}_3(n,x^+\ssb)\ssf
&=\ssb\int_{\epsilon\ssf U}\ssb e^{\ssf(n+2)\ssf\Psi(\theta)}\,
\bigl\{\ssf a_2(\theta)\ssb-\ssb a_2(0)\bigr\}\,d\ssf\theta\,,\\
\widetilde{J}_4(n,x^+\ssb)\ssf&=\ssb\int_{\ssf U\smallsetminus\ssf\epsilon\ssf U}
\bigl[\,\tfrac{h(s\ssf+\ssf i\ssf\theta)}{h(s)}\,
e^{-\ssf i\ssf\langle\ssf\delta,\ssf\theta\ssf\rangle}\ssf\bigr]^{n+2}\;
a_2(\theta)\,d\ssf\theta\,.
\end{aligned}\end{equation*}
\vspace{-.5mm}

Here \ssf$\epsilon\!\in\!(0,1)$ is chosen small enough,
so that Lemma \ref{LocalBehaviorPsi}.(b) holds
for \ssf$\theta\!\in\!\epsilon\ssf U$\ssb.
The first term in \eqref{DecompositionJtilde}
which yields the main contribution,
is estimated as follows.
On the one hand,
according to Lemma \ref{LemmaAmplitudes}.(b),
we have \,$|\ssf a_2(0)|\ssb\approx\ssb1$\ssf,
provided that \ssf$x$ \ssf is large enough.
On the other hand,
we deduce from Lemma \ref{LocalBehaviorPsi} that
\begin{equation*}
\int_{\epsilon\ssf U}\ssb
e^{\ssf(n+2)\Re\Psi(\theta)}\,d\ssf\theta\,
\approx\int_{\epsilon\ssf U}\ssb
e^{\ssf-\const\ssf(n+2)\ssf B(\theta,\theta)}\,d\ssf\theta\,
\approx\int_{\epsilon\ssf U}\ssb
e^{\ssf-\const n\ssf|\theta|^2}\ssf d\ssf\theta\,
\approx\,\tfrac1n\,.
\end{equation*}
Hence
\vspace{-.5mm}
\begin{equation*}
|\ssf\widetilde{J}_1(n,x^+\ssb)|\ssf\approx\ssf\tfrac1n\,.
\end{equation*}
\vspace{-4.5mm}

As
\begin{equation*}
e^{\,i\ssf(n+2)\Im\Psi(\theta)}\!-\ssb1\ssf=\ssf\O\bigl(n\ssf|\theta|^3\bigr)
\qquad\text{and}\qquad
a_2(\theta)\ssb-\ssb a_2(0)=\ssf\O\ssf(\ssf|\theta|\ssf)\,,
\end{equation*}
the next two terms in \eqref{DecompositionJtilde}
are estimated similarly from above\,:
\begin{equation*}
|\ssf\widetilde{J}_2(n,x^+\ssb)|\ssf\lesssim\ssf n^{-\frac32}
\qquad\text{and}\qquad
|\ssf\widetilde{J}_3(n,x^+\ssb)|\ssf\lesssim\ssf n^{-\frac32}\,.
\end{equation*}
For the last term in \eqref{DecompositionJtilde},
we obtain
\begin{equation*}
|\ssf\widetilde{J}_4(n,x^+\ssb)\ssf|\lesssim e^{\ssf-\const n}
\end{equation*}
by using again the Gaussian estimate \eqref{Gaussian}
and the uniform boundedness of \ssf$a_2$\ssf.
Thus,
\begin{equation*}
\widetilde{J}(n,x^+\ssb)\ssf
\ge\,|\ssf\widetilde{J}_1(n,x^+\ssb)\ssf|
-\sum\nolimits_{\ssf k=2}^{\,4}|\ssf\widetilde{J}_k(n,x^+\ssb)\ssf|\,
\gtrsim\ssf\tfrac1n\,,
\end{equation*}
provided that \ssf$n$ \ssf and \ssf$x$ \ssf are large enough.

In summary, we have obtained the following estimates\,:
\begin{equation}\label{UpperLowerBounds}
\begin{cases}
\;p_{\ssf n}(x)\le\ssf\widetilde{p}_{\ssf n}(x)
\lesssim\tfrac1n\,C(n,x^+\ssb)\,,\\
\;p_{\ssf n}(x)\gtrsim\ssf\widetilde{p}_{\ssf n-2}(x)
\gtrsim\tfrac1{n-2}\,C(n\!-\!2,x^+\ssb)\,
\end{cases}\end{equation}
In order to conclude,
observe that the right hand sides in \eqref{UpperLowerBounds} are comparable
and more precisely the Gaussian type factors entering these expressions.
Setting \,$\widetilde{\delta}=\smash{\frac{x^+\ssb+\ssf\rho}n}$ \ssf and writing
\begin{equation*}
\phi(\delta)-\ssf\phi(\widetilde{\delta}\ssf)
=\,-\int_{\frac1{n+2}}^{\ssf\frac1n}\ssb
\langle\ssf d\phi\bigl(t\ssf(x^+\hspace{-1mm}+\!\rho)\bigr),\ssf x^+\hspace{-1mm}+\!\rho\ssf\rangle\,dt\,,
\end{equation*}
we deduce indeed from \eqref{dphi} that
\,$0\ssb\le\ssb\phi(\delta)\ssb-\phi(\widetilde{\delta}\ssf)\ssb\lesssim\ssb\tfrac1n$\ssf,
hence
\begin{equation}\label{FirstComparisonGaussianTypeFactors}
e^{\ssf(n\vsf+\vsf2\vsf)\ssf\phi(\delta)}\approx\ssf e^{\msf n\msf\phi(\widetilde{\delta}\ssf)}\ssf.
\end{equation}

\subsection{Proof of Theorem \ref{TheoremHeatEstimateA2}
when \ssf$\frac{|x|}n$ \ssf gets close to 1,
while \ssf$n\ssb-\ssb|x|$ \ssf and \ssf$x_1\!-\ssb x_2$ remain large.}
\label{Proof4}

This is the most delicate range to handle.
Assume that \ssf$1\!-\ssb\eta\ssb<\!|\delta|\!<\!1$\ssf,
with \,$\eta\!\in\!(0,1)$ \ssf small.
Then \eqref{OptimalBoundA2} amounts to the estimate
\begin{equation}\label{SecondEstimateJnx}
J(n,x^+\ssb)\approx(\ssf n\ssb-\!|x|\ssf)^{-\frac12}\ssf(\ssf n\ssb-\ssb x_1)^{-\frac12}
\end{equation}
for the integral \eqref{DefinitionJnx}.
\smallskip

The upper bound in \eqref{SecondEstimateJnx} is proved as in Subsection \ref{Proof3}.
Firstly, we deduce from \eqref{ComparisonHeatKernels} that
$J(n,x^+\ssb)\ssb\lesssim\ssb\smash{\widetilde{J}}(n,x^+\ssb)$\ssf.
Secondly, the following global estimate is obtained
by combining Lemma \ref{GlobalEstimatePsi},
Lemma \ref{LocalBehaviorPsi}.(a),
Lemma \ref{PropertiesShiftA2}.(e)
and Lemma \ref{PropertiesShiftA2}.(f)\,:
\begin{equation*}
\bigl|\ssf\tfrac{h(s\ssf+\ssf i\ssf\theta)}{h(s)}\ssf\bigr|^{\ssf n+2}
\le\ssf e^{\ssf-\const\ssf(n+2)\ssf B(\theta,\theta)}
\qquad\forall\;\theta\!\in\!U\ssf,
\end{equation*}
\vspace{-5mm}

with
\begin{equation}\label{BehaviorB}
(n\ssb+\ssb2)\ssf B(\theta,\theta)\ssf\approx\ssf
(\ssf n\ssb-\ssb x_1)\ssf(\theta^1\hspace{-.75mm}-\ssb\theta^2)^2\ssb
+(\ssf n\ssb-\!|x|\ssf)\ssf(\theta^1\hspace{-.75mm}+\ssb\theta^2)^2\,.
\end{equation}
\vspace{-3.5mm}

Thirdly,
recall from Lemma \ref{LemmaAmplitudes}.(b) that
the amplitude \eqref{Amplitude2} is uniformly bounded.
Hence the upper bound
\vspace{-.5mm}
\begin{align*}
J(n,x^+\ssb)\ssf&\lesssim\ssf\widetilde{J}(n,x^+\ssb)\\
&\lesssim\int_{\ssf U}\ssb
e^{\ssf-\const\ssf\left\{(n-|x|\hspace{.1mm})(\theta^1\ssb+\ssf\theta^2)^2
+\ssf(n-x_1)(\theta^1\ssb-\ssf\theta^2)^2\right\}}\,d\ssf\theta\\
&\lesssim\ssf(\ssf n\ssb-\hspace{-.4mm}|x|\ssf)^{-\frac12}\ssf
(\ssf n\ssb-\ssb x_1)^{-\frac12}\,.
\end{align*}

Let us turn to the lower bound in \eqref{SecondEstimateJnx},
which is harder to prove.
Our task consists in analyzing the integral
\vspace{1mm}

\centerline{\hfill$\displaystyle
J(n,x^+\ssb)\,=\int_{\ssf U}\,
\bigl[\,\tfrac{h(s\ssf+\ssf i\ssf\theta)}{h(s)}\,
e^{-\ssf i\ssf\langle\ssf\delta,\ssf\theta\ssf\rangle}\ssf\bigr]^{n+2}\,
a_3(\theta)\,d\ssf\theta\,,
$\hfill\eqref{DefinitionJnx}}

which cannot be replaced anymore by the simpler expression
\ssf$\widetilde{J}(n\hspace{-.5mm}-\hspace{-.5mm}2,x^+\ssb)$\ssf,
as the crucial estimate 
\,$\smash{e^{\ssf(n\vsf+\vsf2\vsf)\ssf\phi(\delta)}}\ssb
\lesssim\smash{e^{\msf n\msf\phi(\widetilde{\delta}\ssf)}}$
\ssf fails to hold in \eqref{FirstComparisonGaussianTypeFactors},
when \ssf$s$ \ssf becomes un\-bound\-ed.
For this purpose,
let us split up
\vspace{-.5mm}
\begin{equation}\label{DecompositionJ}
J(n,x^+\ssb)\ssf=\ssf\sum\nolimits_{\ssf k=1}^{\,5}J_k(n,x^+\ssb)\,,
\end{equation}
\vspace{-4mm}

where (see Figure \ref{decomposition})
\begin{equation*}\begin{aligned}
J_1(n,x^+\ssb)\ssf
&=\,a_3(0)\int_{\epsilon\ssf U}\ssb
e^{\ssf(n+2)\Re\Psi(\theta)}\,d\ssf\theta\,,\\
J_2(n,x^+\ssb)\ssf
&=\,a_3(0)\int_{\epsilon\ssf U}
\bigl\{\ssf e^{\ssf(n+2)\ssf\Psi(\theta)}\!
-\ssb e^{\ssf(n+2)\Re\Psi(\theta)}\bigr\}\,d\ssf\theta\,,\\
J_3(n,x^+\ssb)\ssf
&=\ssb\int_{\epsilon\ssf U}\ssb e^{\ssf(n+2)\ssf\Psi(\theta)}\,
\{\ssf a_3(\theta)\ssb-\ssb a_3(0)\}\,d\ssf\theta\,,\\
J_4(n,x^+\ssb)\ssf
&=\ssb\int_{\ssf U\ssf\cap\,\epsilon\ssf(S\ssf\smallsetminus\ssf U)}
\bigl[\,\tfrac{h(s\ssf+\ssf i\ssf\theta)}{h(s)}\,
e^{-\ssf i\ssf\langle\ssf\delta,\ssf\theta\ssf\rangle}\ssf\bigr]^{n+2}\,
a_3(\theta)\,d\ssf\theta\,,\\
J_5(n,x^+\ssb)\ssf
&=\ssb\int_{\ssf U\smallsetminus\ssf\epsilon\ssf(S\ssf\cup\ssf U)}
\bigl[\,\tfrac{h(s\ssf+\ssf i\ssf\theta)}{h(s)}\,
e^{-\ssf i\ssf\langle\ssf\delta,\ssf\theta\ssf\rangle}\ssf\bigr]^{n+2}\,
a_3(\theta)\,d\ssf\theta\,.
\end{aligned}\end{equation*}

\begin{figure}[ht]
\begin{center}
\psfrag{U}[c]{$U$}
\psfrag{S}[c]{$S$}
\psfrag{epsilonU}[c]{\color{red}$\epsilon\ssf U$}
\psfrag{epsilonS}[c]{\color{red}$\epsilon\ssf S$}
\psfrag{2pilambda1}[l]{$2\ssf\pi\lambda_1$}
\psfrag{2pilambda2}[r]{$2\ssf\pi\lambda_2$}
\psfrag{2pilambda1lambda2}[l]{$2\ssf\pi\ssf(\lambda_1\!-\ssb\lambda_2)$}
\includegraphics[width=69mm]{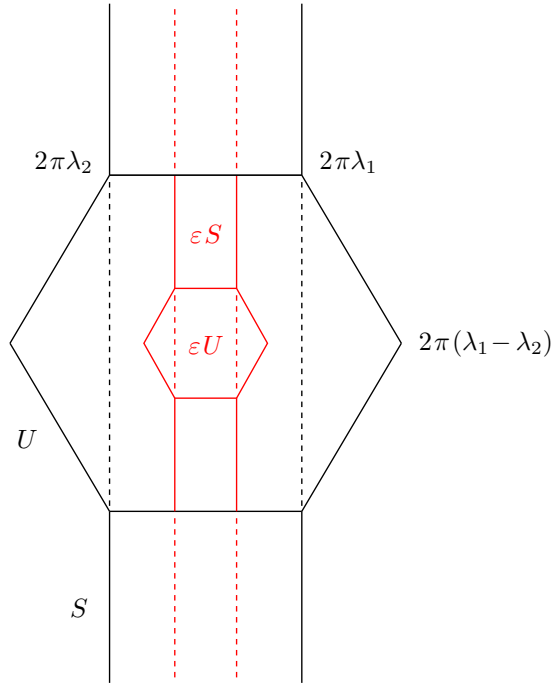}
\end{center}
\caption{Picture for the decomposition \eqref{DecompositionJ}}
\label{decomposition}
\end{figure}

Here, \ssf$S$ \ssf denotes the vertical strip
\vspace{-.5mm}
\begin{equation*}
\{\,\theta\!\in\!\apartment\mid
|\langle\ssf\lambda_1\hspace{-.75mm}-\!\lambda_2,\theta\ssf\rangle|
\!\le\!\tfrac23\ssf\pi\,\}
\end{equation*}
\vspace{-5mm}

and \ssf$\epsilon\!\in\!(0,1)$ is chosen small enough, so that
\begin{description}[labelindent=4pt,labelwidth=4mm,labelsep*=1pt,leftmargin=!]
\item[\textbullet]
Lemma \ref{LocalBehaviorPsi}.(b) holds for \ssf$\theta\!\in\!\epsilon\ssf U$,
\item[\textbullet]
$\ssf a_3(\theta)$ and \ssf$d\ssf a_3(\theta)$
are uniformly bounded for \ssf$\theta\!\in\!\epsilon\ssf(S\ssf\cup\ssf U)$
\ssf and \ssf$n$ \ssf large,
according to Lemma \ref{LemmaAmplitudes}.(c).
\end{description}
We use again \eqref{BehaviorB} to estimate the five integrals \ssf$J_k(n,x^+\ssb)$ occurring in \eqref{DecompositionJ}.
Let us elaborate.
By arguing as for the first three terms in \eqref{DecompositionJtilde}, we obtain
\vspace{-.5mm}
\begin{align*}
|\ssf J_1(n,x^+\ssb)|\ssf&\approx\ssf
(\ssf n\ssb-\ssb x_1)^{-\frac12}\ssf
(\ssf n\ssb-\!|x|\ssf)^{-\frac12}\ssf,\\
|\ssf J_2(n,x^+\ssb)|\ssf&\lesssim\ssf
(\ssf n\ssb-\ssb x_1)^{-\frac12}\ssf
(\ssf n\ssb-\!|x|\ssf)^{-1}\ssf,\\
|\ssf J_3(n,x^+\ssb)|\ssf&\lesssim\ssf
(\ssf n\ssb-\ssb x_1)^{-\frac12}\ssf
(\ssf n\ssb-\!|x|\ssf)^{-1}\ssf,
\end{align*}
provided that \ssf$\eta$ \ssf is small enough
and \ssf$n$ \ssf is large enough,
which we assume from now on.
We obtain similarly
\vspace{-.5mm}
\begin{align*}
|\ssf J_4(n,x^+\ssb)|\ssf\lesssim\ssf
(\ssf n\ssb-\ssb x_1)^{-\frac12}\,
e^{\ssf-\const\ssf(n\ssf-\ssf|x|\hspace{.1mm})}
\end{align*}
\vspace{-4.5mm}

by using the uniform boundedness
of \ssf$a_3(\theta)$ for \ssf$\theta\!\in\!\epsilon\ssf S$
and the estimate
\vspace{-.5mm}
\begin{equation*}
\bigl[\ssf\tfrac{|\ssf h(s\ssf+\ssf i\ssf\theta)\ssf|}{h(s)}\ssf\bigr]^{n+2}
\le\ssf e^{\ssf-\const\ssf(n\ssf-\ssf x_1)(\theta^1\ssb-\ssf\theta^2)^2}
e^{\ssf-\const\ssf(n\ssf-\ssf|x|\hspace{.1mm})}
\qquad\forall\;\theta\ssb\in\ssb
U\ssb\cap\epsilon\ssf(S\ssb\smallsetminus\ssb U\ssf)\,,
\end{equation*}
\vspace{-4mm}

which follows from Lemma \ref{GlobalEstimatePsi} and \eqref{BehaviorB}.
For the last integral, which is most troublesome, we use the estimate
\vspace{-.5mm}
\begin{equation*}
\bigl[\ssf\tfrac{|\ssf h(s\ssf+\ssf i\ssf\theta)\ssf|}{h(s)}\ssf\bigr]^{n+2}
\le\ssf e^{-\const\ssf(n-x_1)}
\qquad\forall\;\theta\ssb\in\ssb U\hspace{-.5mm}\smallsetminus\ssb\epsilon\ssf S\,,
\end{equation*}
\vspace{-4mm}

which follows again from Lemma \ref{GlobalEstimatePsi} and \eqref{BehaviorB},
together with the estimates of \ssf$a_3(\theta)$
contained in Lemma \ref{LemmaAmplitudes}.(c).
On the one hand,
if \ssf$\delta_1\!-\ssb\delta_2$ stays away from \ssf$0$\ssf,
say \ssf$\delta_1\!-\ssb\delta_2\ssb\ge\ssb\frac14$\ssf,
then \ssf$a_3(\theta)$ is uniformly bounded,
hence
\vspace{-1mm}
\begin{equation*}
|\ssf J_5(n,x^+\ssb)|\ssf\lesssim\ssf e^{\ssf-\const\ssf(n\ssf-\ssf x_1)}\ssf.
\end{equation*}
\vspace{-4.5mm}

On the other hand,
if \ssf$\delta_1\!-\ssb\delta_2\hspace{-.5mm}<\hspace{-.5mm}\frac14$\ssf,
then
\vspace{-1mm}
\begin{equation*}
x_1=\tfrac{x_1+\ssf x_2\vphantom{\frac00}}{2\vphantom{\frac00}}\ssb
+\ssb\tfrac{x_1-\ssf x_2\vphantom{\frac00}}{2\vphantom{\frac00}}
<\tfrac{n\ssf-\hspace{.1mm}1\vphantom{\frac00}}{2\vphantom{\frac00}}\ssb
+\ssb\tfrac{n\ssf+\ssf2\vphantom{\frac00}}{4\vphantom{\frac00}}
=\tfrac34\,n\,,
\end{equation*}
\vspace{-4.5mm}

hence \ssf$n\ssb-\ssb x_1\hspace{-.5mm}\approx\ssb n$\ssf.
As \,$a_3(\theta)\ssb=\ssb\O\ssf(n^4)$\ssf,
under the additionnal assumption that \ssf$x_1\!-\ssb x_2$ is large enough,
we obtain again
\vspace{-.75mm}
\begin{equation*}
|\ssf J_5(n,x^+\ssb)|\ssf\lesssim\ssf e^{\ssf-\const\ssf(n\ssf-\ssf x_1)}\ssf.
\end{equation*}
\vspace{-4.5mm}

In conclusion, we obtain the expected bound
\vspace{-1mm}
\begin{equation*}
J(n,x^+\ssb)\ssf
\ge\ssf|\ssf J_1(n,x^+\ssb)|
-\sum\nolimits_{\ssf k=2}^{\,5}|\ssf J_k(n,x^+\ssb)|\ssf
\gtrsim\ssf(\ssf n\ssb-\ssb x_1)^{-\frac12}\ssf
(\ssf n\ssb-\!|x|\ssf)^{-\frac12}\ssf,
\end{equation*}
\vspace{-4mm}

provided that \ssf$\eta$ \ssf is small enough
and that \ssf$n$\ssf, \ssf$n\!-\!|x|$\ssf, \ssf$x_1\!-\ssb x_2$
\ssf are all large enough.

\subsection{Completion of the proof of Theorem \ref{TheoremHeatEstimateA2}
when \ssf$\frac{|x|}n$ \ssf gets close to 1,
while \ssf$n\ssb-\ssb|x|$ \ssf remains large.}
\label{Proof5}

In this subsection, we extend up to the extra wall \eqref{ExtraWall}
the lower bound proved in Subsection \ref{Proof4}.
Specifically, assume that
the lower bound in \eqref{OptimalBoundA2} holds in the range

\centerline{$\begin{cases}
\;x_1\ssb-x_2\ge m\,,\\
\;n-|x|\ge m\,,
\end{cases}$}
\vspace{.5mm}

for some fixed \ssf$m\!\in\!\N^*$,
and let us deduce it under the following conditions\,:
\vspace{1mm}

\centerline{$\begin{cases}
\;0\le x_1\ssb-x_2\ssb<m\,,\\
\;n-|x|\ge m\,,\\
\;1\!-\ssb\eta<\frac{|x|\vphantom{\frac00}}n<1\ssf,
\text{ \ssf with \,$0\ssb<\ssb\eta\ssb<\ssb1$ \,small enough\ssf,}\\
\;\text{$n$ \,or equivalently \,$|x|$ \,is large enough\ssf.}
\end{cases}$}
\vspace{.5mm}

In this case, \ssf$x_1$ and \ssf$x_2$
\hspace{.1mm}are close to \ssf$\frac{|x|}2$\ssf.
Thus the lower estimate in \eqref{OptimalBoundA2},
which we aim for, amounts to
\vspace{-1mm}
\begin{equation}\label{LowerEstimateExtraWall}
p_{\ssf n}(x)\gtrsim\ssf
n^{-\frac12}\ssf(\ssf n\ssb-\!|x|\ssf)^{-\frac12}\;
\boldsigma^{\vsf n}\,
q^{-\langle\ssf\rho,\ssf x^+\rangle}\ssf
e^{\ssf n\ssf\phi(\delta)}\ssf.
\end{equation}
Consider \,$\widetilde{n}\ssb=\ssb n\ssb-\ssb m$ \,and
\,$\widetilde{x}\ssb=\ssb x^+\hspace{-.8mm}-\ssb m\ssf\lambda_2$\ssf.
Then
\vspace{1mm}

\centerline{$
p_{\ssf n}(x)\gtrsim\ssf p_{\ssf\widetilde{n}}(\widetilde{x})\ssf,
$}\vspace{1mm}

according to the local Harnack inequality \eqref{Harnack2}.
Besides \ssf$\widetilde{x}\!\in\!\sector\hspace{-.4mm}$
provided that \ssf$|x|$ \ssf is large enough.
Moreover

\centerline{$\begin{cases}
\;\widetilde{x}_1\ssb-\widetilde{x}_2=\ssf x_1\ssb-x_2+m\ge m\,,\\
\;\widetilde{n}-|\widetilde{x}|=\ssf n-|x|\ge m\,.\\
\end{cases}$}
\vspace{1mm}

Thus \eqref{LowerEstimateExtraWall} holds for
\ssf$p_{\ssf\widetilde{n}}(\widetilde{x})$ by assumption,
hence
\vspace{-1mm}
\begin{equation*}
p_{\ssf\widetilde{n}}(\widetilde{x})\gtrsim\ssf
n^{-\frac12}\ssf(\ssf n\ssb-\!|x|\ssf)^{-\frac12}\;
\boldsigma^{\vsf n}\,
q^{-\langle\ssf\rho,\ssf x^+\rangle}\ssf
e^{\ssf\widetilde{n}\ssf\phi(\widetilde{\delta})}\ssf,
\end{equation*}
\vspace{-4.5mm}

where \,$\widetilde{\delta}\ssb
=\ssb\frac{\widetilde{x}\ssf+\ssf\rho\vphantom{|}}
{\widetilde{n}\ssf+\ssf2\vphantom{|}}$\ssf.
In order to conclude the proof of \eqref{LowerEstimateExtraWall},
it remains for us to compare
\vspace{-2.5mm}
\begin{equation*}\label{SecondComparisonGaussianTypeFactors}
e^{\ssf\widetilde{n}\ssf\phi(\widetilde{\delta})}
\gtrsim\ssf e^{\ssf n\ssf\phi(\delta)}\,.
\end{equation*}
\vspace{-4mm}

Firstly,
\,$\widetilde{n}\,\phi(\widetilde{\delta})\ssb
\ge n\,\phi(\widetilde{\delta})$\ssf,
as \ssf$\phi\ssb\le\ssb0$\ssf.
Secondly, we claim that
\,$\phi(\widetilde{\delta}\ssf)\hspace{-.5mm}\ge\ssb\phi(\delta)$\ssf,
\ssf if \ssf$\eta$ \ssf is small enough and \ssf$n$ \ssf large enough.
For this purpose, let us write
\vspace{-.5mm}
\begin{equation}\label{AuxiliaryExpression1}
\phi(\widetilde{\delta}\ssf)-\ssf\phi(\delta)\ssf
=\ssb\int_{\,0}^{\ssf1}\ssb\langle\ssf
(d\ssf\phi\ssb\circ\ssb\delta)(t)\ssf,
\ssf\widetilde{\delta}\ssb-\ssb\delta\,\rangle\,dt\,,
\end{equation}
\vspace{-2.5mm}

where \,$\delta(t)\ssb
=\ssb(1\hspace{-1mm}-\!t)\ssf\delta\ssb+\ssb t\ssf\widetilde{\delta}$\ssf.
In this expression,
\ssf$(d\ssf\phi\ssb\circ\ssb\delta)(t)\ssb=\ssb-\,s(t)$\ssf,
according to \eqref{dphi}, and
\begin{equation*}
\widetilde{\delta}\ssb-\ssb\delta
=\tfrac{m\vphantom{|}}
{(n\ssf+\ssf2)\ssf(n\ssf+\ssf2\ssf-\ssf m)\vphantom{|}}\,
\bigl\{\ssf(\hspace{.1mm}x_1\!+\hspace{-.5mm}1)\,\lambda_1\ssb
+(\hspace{.1mm}x_2\!-\ssb n\ssb-\!1)\,\lambda_2\ssf\bigr\}\,.
\end{equation*}
\vspace{-4.5mm}

Hence,
\begin{equation}\label{AuxiliaryExpression2}
\langle\ssf(d\ssf\phi\ssb\circ\ssb\delta)(t),
\ssf\widetilde{\delta}\ssb-\ssb\delta\,\rangle
=\tfrac{m\vphantom{|}}
{(n\ssf+\ssf2)\ssf(n\ssf+\ssf2\ssf-\ssf m)\vphantom{|}}\,
\bigl\{\ssf(\hspace{.1mm}n\hspace{-.4mm}-\hspace{-.5mm}|x|)\,s^2(t)\ssb
-\ssb(\hspace{.1mm}x_1\!+\!1)\,
[\ssf s^1\hspace{-.1mm}(t)\hspace{-.5mm}-\hspace{-.5mm}s^2(t)\ssf]\ssf\bigr\}\,.
\end{equation}
On the one hand,
\begin{equation*}
\delta_1(t)\ssb-\delta_2(t)
=(1\!-\ssb t\ssf)\,\tfrac{x_1-\ssf x_2\vphantom{|}}{n\ssf+\ssf2\vphantom{|}}
+t\,\tfrac{x_1-\ssf x_2\hspace{.1mm}+\ssf m\vphantom{|}}
{n\ssf+\ssf2\ssf-\ssf m\vphantom{|}}
=\ssf\O\bigl(\tfrac{1\vphantom{|}}n\bigr)\,.
\end{equation*}
Thus, if \ssf$n$ \ssf is large enough,
\begin{equation}\label{AuxiliaryExpression3}
(\hspace{.1mm}x_1\!+\!1)\,
[\ssf s^1\hspace{-.1mm}(t)\hspace{-.5mm}-\hspace{-.5mm}s^2(t)\ssf]
=\ssf\O\ssf(1)\,,
\end{equation}
according to Lemma \ref{PropertiesShiftA2}.(g).
On the other hand,
\vspace{-.5mm}
\begin{equation*}
1-|\ssf\delta(t)|=(1\!-\ssb t\ssf)\,
\tfrac{n\ssf-\ssf|x|\vphantom{|}}{n\ssf+\ssf2}
+t\,\tfrac{n\ssf-\ssf|x|\vphantom{|}}{n\ssf+\ssf2\ssf-\ssf m}
\le\tfrac{n\ssf-\ssf|x|\vphantom{|}}{n\ssf+\ssf2\ssf-\ssf m}
=\left(1\!-\ssb\smash{\tfrac{m-2\vphantom{|}}n}\vphantom{\tfrac oo}\right)^{-1}
\left(1\!-\ssb\smash{\tfrac{|x|\vphantom{|}}n}\vphantom{\tfrac oo}\right)
\end{equation*}
is smaller than \ssf$2\ssf\eta$\ssf,
if \ssf$n$ \ssf is large enough,
and
\vspace{-.5mm}
\begin{equation*}
e^{\ssf s^2\hspace{-.1mm}(t)}\approx\ssf
\tfrac{1\vphantom{|}}{1\ssf-\ssf|\delta(t)|\vphantom{\frac00}}\,,
\end{equation*}
\vspace{-4.5mm}

according to Lemma \ref{PropertiesShiftA2}.(e).
Thus,
\begin{equation*}
(\ssf n\hspace{-.4mm}-\hspace{-.5mm}|x|\ssf)\,s^2(t)
\ge m\hspace{.5mm}\bigl[\ssf-\log\ssf(2\ssf\eta)\vsb-\const\ssf\bigr]
\end{equation*}
is positive and even larger than \eqref{AuxiliaryExpression3},
if \ssf$\eta$ \ssf is small enough.
In conclusion, \eqref{AuxiliaryExpression2} is positive
and \eqref{AuxiliaryExpression1} too,
provided that \ssf$\eta$ \ssf is small enough
and \ssf$n$ \ssf large enough.

\subsection{Proof of Theorem \ref{TheoremHeatEstimateA2}
close to the boundary \,$|x|\!=\ssb n\,$}
\label{Proof6}

In this subsection,
we give a combinatorial proof of Theorem \ref{TheoremHeatEstimateA2}
in the range \ssf$n\ssb-\ssb m\ssb\le\ssb|x|\ssb\le\ssb n$\ssf,
where \ssf$m$ \ssf is any fixed positive integer and \ssf$n$ \ssf is large.
Still assuming  that \ssf$x_1\!\ge\ssb x_2$\ssf,
let us first show that \eqref{OptimalBoundA2} amounts then to
\vspace{-.5mm}
\begin{equation}\label{HeatEstimateA2Boundary}
p_{\ssf n}(x)\approx\sigma^{\ssf n}\,q^{-n}\,n^{n+d}\,
x_1^{-\ssf x_1}\,(\ssf x_2\!+\!1)^{-\ssf x_2-\ssf d\ssf-\frac12}\ssf,
\end{equation}
where \ssf$d\ssb=\ssb n\ssb-\ssb|x|$\ssf.
Firstly,
\begin{equation*}
\tfrac{(1\ssf+\ssf|x|\hspace{.1mm})\ssf
(1\ssf+\ssf x_1)\ssf(1\ssf+\ssf x_2)\vphantom{\big|}}
{n^{3\vphantom{\frac oo}}\ssf\sqrt{\ssf n\ssf-\ssf|x|\,}\ssf
\sqrt{\ssf n\ssf-\ssf x_1\ssf\vphantom{|}}}\ssf
\approx\tfrac{\sqrt{x_2+1\vphantom{|}}\vphantom{\big|}}{n\vphantom{\sqrt{|}}}
\end{equation*}
as \ssf$x_1\!\approx\ssb|x|\ssb\approx\ssb n$
\ssf and \ssf$n\ssb-\ssb x_1\!=x_2\!+\ssb d\approx x_2\!+\!1$\ssf.
Secondly,
\begin{equation*}
\widetilde\sigma^{\ssf n}\ssb=6^{\ssf n}\ssf\sigma^{\ssf n}
\qquad\text{and}\qquad
q^{-\langle\rho,\ssf x^{\ssb+\ssb}\rangle}\ssb=q^{-|x|}\ssb\approx q^{-n}\,.
\end{equation*}
Thirdly,
\begin{equation*}
e^{\ssf n\ssf\phi(\delta)}\ssb
\approx e^{\ssf(n+2)\ssf\phi(\delta)}\ssb
\approx 6^{-n}\,h(s)^{n+2}\,
e^{-\langle\ssf x^+\ssb+\ssf\rho,\ssf s\ssf\rangle}\,,
\end{equation*}
with \,$h(s)\ssb=e^{\ssf s^1}
\bigl\{\ssf1\ssb+e^{-(s^1\!-s^2)}\!
+\O\bigl(e^{-s^2}\bigr)\ssf\bigr\}$\ssf.
According to Lemma \ref{PropertiesShiftA2},
\begin{equation*}
h(s)=\ssf e^{\ssf s^1}
\bigl\{\ssf\tfrac1{\delta_1}\ssb
+\O\ssf(\ssf1\!-\ssb|\delta|\ssf)\ssf\bigr\}
=\ssf\tfrac{e^{\ssf s^1}}{\delta_1}\ssf
\bigl\{\ssf1\ssb+\O\ssf(\tfrac1n)\ssf\bigr\}\,,
\end{equation*}
hence
\begin{equation*}
h(s)^{n+2}\,e^{-\langle\ssf x^+\ssb+\ssf\rho,\ssf s\ssf\rangle}
\approx\ssf n^{n+1}\,x_1^{-n-1}\,e^{\ssf(x_2+\ssf d\ssf+1)\ssf(s^1\ssb-s^2)}\,
e^{\ssf d\ssf s^2}\ssf,
\end{equation*}
with \,$e^{\ssf d\vsf s^2}\!\approx\ssb(1\!-\!|\delta|)^{-d}\ssb
\approx n^{\ssf d}$\ssf.
On the one hand, if \ssf$x_2$ is large enough,
\begin{align*}
e^{\ssf(x_2+\ssf d\ssf+\vsf1)\ssf(s^1\ssb-s^2)}
&\approx\,
\bigl\{\ssf\tfrac{1\ssf-\ssf\delta_1}{\delta_1}
+\O\ssf(\ssf1\!-\ssb|\delta|\ssf)\ssf\bigr\}^{-x_2-d-1}\\
&\approx\,
\bigl(\tfrac{1\ssf-\ssf\delta_1}{\delta_1}\bigr)^{-x_2-d-1}\,
\bigl\{\ssf1\ssb+\O\bigl(\tfrac1{x_2\vsf+\ssf d\ssf+\vsf1}\bigr)\bigr\}^{-x_2-d-1}\\
&\approx\,
\bigl(\tfrac{x_1+\vsf1}{x_2\vsf+\ssf d\ssf+\vsf1}\bigr)^{x_2+d+1}\ssf
\approx\,\bigl(\tfrac{x_1}{x_2\vsf+\vsf1}\bigr)^{x_2+d+1}\,.
\end{align*}
On the other hand, as long as  \ssf$x_2$ is bounded,
\begin{equation*}
e^{\ssf(x_2+\ssf d\ssf+1)\ssf(s^1\ssb-s^2)}
\approx\bigl(\tfrac{1\ssf-\ssf\delta_1}{\delta_1}\bigr)^{-x_2-d-1}
=\bigl(\tfrac{x_1+1}{x_2+\ssf d\ssf+\ssf1}\bigr)^{x_2+d+1}
\approx\bigl(\tfrac{x_1}{x_2+1}\bigr)^{x_2+d+1}\,.
\end{equation*}
\vspace{-4.5mm}

Thus
\vspace{-.5mm}
\begin{equation*}
h(s)^{n+2}\,
e^{-\langle\ssf x^+\ssb+\ssf\rho,\ssf s\ssf\rangle}
\approx\ssf n^{\ssf n+d+1}\,
x_1^{-x_1}\,(\ssf x_2\!+\!1)^{-x_2-d-1}
\end{equation*}
in all cases and the right hand side of \eqref{OptimalBoundA2}
is comparable to \eqref{HeatEstimateA2Boundary}, as claimed.
\smallskip

Let us next turn to the proof of \eqref{HeatEstimateA2Boundary}.
Instead of the simple random walk in $\X$,
it is more convenient to work with the corresponding radial random walk in $P^+$,
whose transition probability is given by the following table
where, let us recall,
\ssf$\sigma\ssb=\ssb\frac1{2\ssf(q\ssf+1+\ssf q^{-1})}$\ssf.
\medskip

\begin{center}\begin{tabular}{|l|l|}\hline
\;$\lambda\!\in\!P^{++}$
&\begin{tabular}{l}
$p^+(\lambda,\lambda\ssb+\!\lambda_1)
=p^+(\lambda,\lambda\ssb+\!\lambda_2)
=\sigma\ssf q
\vphantom{\overbrace{o}}$\\
$p^+(\lambda,\lambda\ssb+\!\lambda_1\hspace{-.75mm}-\!\lambda_2)
=p^+(\lambda,\lambda\ssb-\!\lambda_1\hspace{-.75mm}+\!\lambda_2)
=\sigma$\\
$p^+(\lambda,\lambda\ssb-\!\lambda_1)
=p^+(\lambda,\lambda\ssb-\!\lambda_2)
=\sigma\ssf q^{-1}
\vphantom{\underbrace{o}}$\\
\end{tabular}\\\hline
\;$\lambda\!\in\!\N^*\ssb\lambda_1$
&\begin{tabular}{l}
$p^+(\lambda,\lambda\ssb+\!\lambda_2)=\sigma\ssf(q\ssb+\!1)
\vphantom{\overbrace{o}}$\\
$p^+(\lambda,\lambda\ssb+\!\lambda_1)=\sigma\ssf q$\\
$p^+(\lambda,\lambda\ssb-\!\lambda_1\hspace{-.75mm}+\!\lambda_2)
=\sigma\ssf(1\!+\ssb q^{-1})$\\
$p^+(\lambda,\lambda\ssb-\!\lambda_1)=\sigma\ssf q^{-1}
\vphantom{\underbrace{o}}$\\
\end{tabular}\\\hline
\;$\lambda\!\in\!\N^*\ssb\lambda_2$\;
&\begin{tabular}{l}
$p^+(\lambda,\lambda\ssb+\!\lambda_1)=\sigma\ssf(q\ssb+\!1)
\vphantom{\overbrace{o}}$\\
$p^+(\lambda,\lambda\ssb+\!\lambda_2)=\sigma\ssf q$\\
$p^+(\lambda,\lambda\ssb+\!\lambda_1\hspace{-.75mm}-\!\lambda_2)
=\sigma\ssf(1\!+\ssb q^{-1})$\\
$p^+(\lambda,\lambda\ssb-\!\lambda_2)=\sigma\ssf q^{-1}
\vphantom{\underbrace{o}}$\\
\end{tabular}\\\hline
\;$\lambda\ssb=\ssb0$
&\begin{tabular}{l}
$p^+(\lambda,\lambda\ssb+\!\lambda_1)
=p^+(\lambda,\lambda\ssb+\!\lambda_2)
=\frac12\vphantom{\overbrace{o}}\vphantom{\underbrace{o}}$
\end{tabular}\\\hline
\end{tabular}\end{center}
\medskip

We claim that,
in the range \,$n\ssb-\ssb m\ssb\le\ssb|x|\ssb\le\ssb n$\ssf,
\eqref{HeatEstimateA2Boundary} amounts to showing that
\begin{equation}\label{NumberPaths}
M\approx\ssf n^{\ssf n+d}\,x_1^{-x_1}\ssf(\ssf x_2\!+\!1)^{-x_2-d-\frac12}\ssf,
\end{equation}
where \ssf$M$ denotes the number of paths in \ssf$P^+$
between \ssf$0$ \ssf and \ssf$\lambda\ssb=\ssb x^+$\ssb.
This claim is obtained by combining the following two facts.
On the one hand,
\,$p_{\ssf n}^+(0,\lambda)\ssb
=\ssb N_{\ssb\lambda}\,p_{\ssf n}(\lambda)$\ssf,
where \ssf$N_{\ssb\lambda}\ssb
\approx q_{\ssf t_\lambda}\hspace{-.75mm}
=q^{\ssf2\ssf\langle\rho,\ssf\lambda\rangle}\ssb
=q^{\ssf2\ssf|\lambda|}\ssb
\approx q^{\ssf2\ssf n}\ssf$.
On the other hand, according to the table above,
the transition probability of the radial random walk
equals \ssf$\sigma\ssf q$ \ssf or \ssf$\sigma\ssf(q\ssb+\!1)$
at each step, except for finitely many.
\smallskip

Let us turn to the proof of \eqref{NumberPaths} and,
for this purpose, consider the sequence of increments
of the radial random walk up to time \ssf$n$\,:
\begin{equation*}
\epsilon=(\epsilon_1,\dots,\epsilon_n)
\quad\text{with}\quad
\epsilon_j\!\in\!\{\pm\lambda_1,\pm\lambda_2\ssf,
\pm\lambda_1\hspace{-1mm}\mp\hspace{-.75mm}\lambda_2\}\,.
\end{equation*}
On the one hand, choose \ssf$x_1\!-\ssb d$ \ssf times \ssf$\lambda_1$,
\ssf$x_2\!+\ssb d$ \ssf times \ssf$\lambda_2$,
and \ssf$d$ \ssf times \ssf$\lambda_1\!-\ssb\lambda_2$,
the latter occurring after at least \ssf$d$ \ssf increments \ssf$\lambda_2$,
in order to remain within \ssf$P^+$\ssb.
The number of such choices equals
\vspace{-.5mm}
\begin{equation*}
\tfrac{n\ssf!\vphantom{|}}
{(x_1-\ssf d\ssf)\ssf!\,(x_2\ssf+\ssf2\ssf d\ssf)\ssf!}\ssf
\tfrac{(x_2\ssf+\ssf d\ssf)\ssf!}{x_2\ssf!\,d\,!}\,,
\end{equation*}
\vspace{-5mm}

which is comparable to
\vspace{-.5mm}
\begin{equation*}
n_{\vphantom{0}}^{\ssf n+\frac12}\,
x_1^{-x_1+d-\frac12}\,
(x_2\!+\!1)_{\vphantom{0}}^{-x_2-d-\frac12}\,,
\end{equation*}
according to Stirling's formula, hence to
\ssf$n_{\vphantom{0}}^{n+d}\ssf
x_1^{-x_1}(x_2\!+\!1)^{-x_2-d-\frac12}$\ssf.
This proves the lower bound.
On the other hand, in order to reach
\ssf$\lambda\ssb=\ssb x_1\lambda_1\!+\ssb x_2\ssf\lambda_2$
in \ssf$n$ \ssf steps,
one needs \ssf$k$ \ssf in\-cre\-ments \ssf$\lambda_1$,
\ssf$\ell$ \ssf increments \ssf$\lambda_2$
and \ssf$n\ssb-\ssb k\ssb-\ssb\ell$ \ssf increments in
\ssf$\{-\lambda_1,-\lambda_2,\pm\ssf\lambda_1\!\mp\ssb\lambda_2\}\ssf$.
Notice that \ssf$|x|\!\le\ssb k\ssb+\ssb\ell\ssb\le\ssb n$
\ssf and \ssf$|\ssf k\ssb-\ssb x_1|\ssb\le\ssb d$\ssf,
\ssf$|\ssf\ell\ssb-\ssb x_2|\ssb\le\ssb d$\ssf.
Thus
\vspace{-.5mm}
\begin{equation}\label{UpperBound}
M\hspace{1mm}\le\hspace{3mm}
\sum\nolimits_{\hspace{-9mm}\substack{\vphantom{o}\\
k,\ssf\ell\ssf\in\ssf\N\\
|\ssf k-x_1|\ssf\le\ssf d\ssf,\,|\ssf\ell-x_2|\ssf\le\ssf d}}
\hspace{-1mm}
\tfrac{n\ssf!}{k\ssf!\,\ell\ssf!}\,.
\end{equation}
According to Stirling's formula,
each term \ssf$\smash{\frac{n\ssf!}{k\ssf!\,\ell\ssf!}}$
\ssf on the right hand side of \eqref{UpperBound} is comparable to
\,$\smash{n^{\ssf n+\frac12}\,k^{-k-\frac12}\ssf(\ell\ssb+\!1)^{-\ell-\frac12}}$
\ssf hence to
\,$\smash{n^{\ssf n}\,x_1^{-k}\ssf(x_2\ssb+\!1)^{-\ell-\frac12}}$.
The upper bound
\vspace{-.5mm}
\begin{equation*}
M\lesssim n^{\ssf n+d}\,x_1^{-x_1}\ssf(x_2\ssb+\!1)^{-x_2-\ssf d\ssf-\frac12}
\end{equation*}
follows from the fact that the sum in \eqref{UpperBound} is finite and
\vspace{-1mm}
\begin{equation*}
x_1^{\ssf x_1\ssb-k}\ssf(x_2\!+\!1)^{\ssf x_2+d-\ell}
\lesssim\ssf n^{\ssf n-k-\ell}\le\ssf n^{\ssf d}\,.
\end{equation*}
\vspace{-3mm}

This concludes the proof of Theorem \ref{TheoremHeatEstimateA2}.
\hfill$\square$

\begin{remark}
Most of the analysis carried out in this section applies actually
to any isotropic nearest neighbor random walk
\vspace{-.5mm}
\begin{equation*}
A=\cone\,A_{\lambda_1}\!+\ctwo\,A_{\lambda_2}\,,
\end{equation*}
\vspace{-4.5mm}

where \,$\cone\!>\ssb0$\ssf, $\ctwo\ssb>\ssb0$ and \,$\cone\!+\ssb\ctwo\ssb=\ssb1$\ssf.
More precisely, such a random walk has transition density
\begin{equation*}
p_{\ssf n}(x)
=\tfrac1{4\ssf\pi^2}\,\sigma^{\ssf n}\,q^{-\langle\ssf\rho,\ssf x^+\rangle}\ssb
\int_{\ssf U}\ssb h(i\ssf\theta)^n\,\tfrac
{\Delta(i\ssf\theta)\,e^{-i\langle x^+\ssb+\rho,\ssf\theta\ssf\rangle}}
{\mathbf{b}(i\ssf\theta)}\,d\ssf\theta
\end{equation*}
with
\begin{equation*}
h\ssf=\ssf2\ssf c_1\sum\nolimits_{\ssf\lambda\in W_0.\lambda_1}\!e^{\ssf\lambda}
+\ssf2\ssf c_2\sum\nolimits_{\ssf\lambda\in W_0.\lambda_2}\!e^{\ssf\lambda}\,.
\end{equation*}
As for the simple random walk, its spectral radius is
\,$\boldsigma\ssb=\ssb6\msf\sigma\ssb=\ssb\frac{3\vphantom{|}}{q\ssf+1+\ssf q^{-1}}$
and \,$h$ enjoys remarkable product and differentiation formulae $($see Appendix \ref{Appendix}\vsf$)$.
Thus the same analysis yields again the upper and lower bound
\begin{equation*}
n^{-4}\,\boldsigma^{\vsf n}\msf F_0(x)\,e^{\ssf n\ssf\phi(\delta)}
\end{equation*}
in the range \,$|x|\msb\le\msb(1\!-\msb\eta)\ssf n$\ssf.
In the range \,$(1\!-\msb\eta)\ssf n\msb<\msb|x|\msb<\ssb n\ssf$, there is a problem with the lower bound,
but we get the upper bound
\begin{equation*}
\tfrac{e^{\ssf C\vsf(n-|x|)}}
{n^{3\vphantom{\frac oo}}\ssf\sqrt{\ssf n\ssf-\ssf|x|\,}\ssf\sqrt{\ssf n\ssf-\ssf x_1\ssb\vee x_2\ssf\vphantom{|}}}\;
\boldsigma^{\vsf n}\msf F_0(x)\,e^{\ssf n\ssf\phi(\delta)}
\end{equation*}
by arguing as in Subsection \ref{ProofUpperEstimateAr}.
Here \,$C$ is a positive constant.
\end{remark}

\section{Heat kernel estimates in higher rank}
\label{sectionAr}

Let us generalize the notation of Section \ref{sectionA2}.
Given an integer \ssf$n\msb\ge\msb2$ \ssf and \ssf$x\!\in\!\X$ with \ssf$|x|\!<\!n$\ssf,
set \ssf$\delta\ssb=\ssb\tfrac{x^+\msb+\ssf\rho}{n\ssf+\ssf r}$
and denote by \ssf$\phi(\delta)\!\in\!(-\ssb\log\vsb N\ssb,0\ssf]$ \ssf the minimum of the function
\begin{equation}\label{DefinitionPhiAr}
\Phi(z)=\ssf\log\tfrac{h(z)}N-\langle\ssf\delta,z\ssf\rangle
\qquad\forall\,z\!\in\!\apartment\ssf,
\end{equation}
where
\begin{equation*}
N\ssb=\ssb h(0)\ssb=\sum\nolimits_{\ssf j=1}^{\,r}|\vsf W_0\vsf.\ssf\lambda_j|=2\ssf(2^{\vsf r}\hspace{-1mm}-\!1)\msf.
\end{equation*}
Here is a partial generalization of Theorem \ref{TheoremHeatEstimateA2}.

\begin{theorem}\label{TheoremHeatEstimateAr}
Let \,$0\msb<\msb\eta\msb<\!1$ $($small$\,)$.
Then
\begin{equation}\label{OptimalBoundAr}
p_{\ssf n}(x)\approx n^{-\frac r2-\vsf|R^+\ssb|}\,\boldsigma^{\ssf n}\ssf F_0(x)\,e^{\ssf n\ssf\phi(\delta)}
\end{equation}
uniformly in the range \,$|x|\msb\le\msb(1\!-\ssb\eta)\ssf n$\ssf.
Moreover,
\begin{equation}\label{UpperBoundAr}
p_{\ssf n}(x)\le n^{-\frac r2-\vsf|R^+\ssb|}\,\boldsigma^{\ssf n}\ssf F_0(x)\,e^{\ssf n\ssf\phi(\delta)}\,
\tfrac{e^{\ssf n\vsf(1-|\delta|)}}{\prod_{\ssf\alpha\in\vsb R^+}\!\sqrt{1-\vsf\langle\alpha,\ssf\delta\rangle}}
\end{equation}
in the range \,$|x|\msb<\ssb n$\vsf.
\end{theorem}

\begin{remark}
Notice that the power \,$\frac r2\msb+\msb|R^+\vsb|$ is half of the pseudo-dimension,
as in the symmetric space case.
\end{remark}

The proof of Theorem \ref{TheoremHeatEstimateAr} is similar to
the proof of Theorem \ref{TheoremHeatEstimateA2} in Section \ref{sectionA2}.
Therefore, we just outline the proof, elaborating on higher rank features.
Let us begin with analogs of \eqref{ProductFormulaA2},
\eqref{DifferentiationFormulaA2} and \eqref{DifferentiationFormulaRankTwo}.

\begin{lemma}\label{LemmaAr}
The following product and differentiation formulae hold\,$:$
\begin{equation}\label{ProductFormulaAr}
h\ssb+\ssb2=\!\prod_{\lambda\in W_0.\lambda_1}\!(e^{\ssf\lambda}\msb+\ssb1)
=\!\prod_{1\le j\le r+1}^r\!(e^{\ssf\lambda_j-\lambda_{j-1}}\!+\ssb1)\msf,
\end{equation}
where we set \,$\lambda_{\ssf0}\msb=\ssb\lambda_{\ssf r+1}\!=\ssb0$\ssf, and
\begin{equation}\label{DifferentiationFormula1Ar}
\boldpi(\partial)\msf h^{\vsf n+|R^+\ssb|}
=\underbrace{\tfrac{(n+|R^+\ssb|)\ssf!}{n\ssf!}}_{\approx\,n^{\vsf|\ssb R^+\ssb|}}\,r_n(h)\,h^n\ssf\Delta\,,
\end{equation}
\vspace{-3mm}

where
\vspace{-3mm}
\begin{equation*}
r_n(h)=(h\msb+\msb2)^{|\vsb R^+\ssb|-r}+\!\sum_{r\vsf\le k<|R^+\ssb|}\!
\overbrace{c_k\msf\tfrac{n\ssf!}{(n+|R^+\ssb|-k)\ssf!}}^{{\rm\mathcal{O}}\vsf\bigl(n^{-(|\ssb R^+\ssb|-k)}\bigr)}\,
(h\msb+\msb2)^{k-r}\,h^{\ssf|R^+\ssb|-k}
\end{equation*}
is a polynomial in \,$h$ \vsf with coefficients \,$c_k\msb\in\ssb\Z$\ssf.
Moreover,
\begin{equation}\label{DifferentiationFormula2Ar}
\boldpi(\partial)\msf(h\msb+\msb2)^{n+r}\ssb=d_{\vsf n}\msf(h\msb+\msb2)^n\ssf\Delta\,,
\end{equation}
where \,$d_{\vsf n}$ is a constant, which is positive and \,$\approx\ssb n^{-|R^+|}$ for \,$n$ large enough.
\end{lemma}

\begin{proof}
Let us first prove \eqref{ProductFormulaAr}.
Recall that the fundamental weights satisfy
\begin{equation*}
\langle\ssf\lambda_j,z\ssf\rangle\ssb=z_1\!+\vsf\dots\vsf+\ssb z_j
\quad\text{hence}\quad
\langle\ssf\lambda_j\msb-\ssb\lambda_{j-1},z\ssf\rangle\ssb=z_j
\end{equation*}
for every \vsf$1\!\le\ssb j\msb\le\ssb r\msb+\!1$ \vsf and \ssf$z\!\in\!\apartment$\ssf.
We deduce on the one hand that
\begin{equation}\label{AuxiliaryEquation1}
\sum\nolimits_{\vsf\lambda\in W_0.\vsf\lambda_j}\msb e^{\ssf\langle\lambda,\ssf z\rangle}
=\sum\nolimits_{\vsf1\le\ssf k_1\ssb<\ssf\dots\ssf<\ssf k_j\le\ssf r+1}\ssb e^{\ssf z_{k_1}\msb+\ssf\dots\ssf+\ssf z_{k_j}}
\end{equation}
and on the other hand that
\begin{equation}\label{AuxiliaryEquation2}
\prod\nolimits_{\ssf\lambda\in W_0.\vsf\lambda_1}\ssb\bigl(\vsf e^{\ssf\langle\lambda,\ssf z\rangle}\!+\msb1\bigr)
=\prod\nolimits_{\vsf1\le\vsf j\le\vsf r+1}\bigl(\vsf e^{\ssf z_j}\!+\msb1\bigr)
=\prod\nolimits_{\vsf1\le\vsf j\le\vsf r+1}\bigl(\vsf e^{\ssf\langle\lambda_j-\lambda_{j-1},z\ssf\rangle}\!+\msb1\bigr)\ssf.
\end{equation}
By adding up \eqref{AuxiliaryEquation1} over \vsf$1\!\le\ssb j\ssb\le\msb r$\vsf,
we obtain that \ssf$h(z)\msb+\ssb2$ \ssf is equal to the sum of products \ssf$\prod_{\ssf j\ssf\in J}e^{\ssf z_j}$,
where $J$ runs through all subsets of \ssf$\{\vsf1\vsf,\ssf\dots,\ssf r\msb+\!1\vsf\}$\ssf,
which is equal in turn to
\begin{equation*}
\prod\nolimits_{\vsf1\le\vsf j\le\vsf r+1}\bigl(\vsf e^{\ssf z_j}\msb+\ssb1\bigr)\ssf.
\end{equation*}
Together with \eqref{AuxiliaryEquation2}, this concludes the proof of \eqref{ProductFormulaAr}.
Let us turn to the proof of \eqref{DifferentiationFormula1Ar}. Setting
\vspace{-1mm}
\begin{equation*}
\boldpi_I(\lambda)=\prod\nolimits_{\ssf\alpha\in I}\langle\ssf\alpha,\lambda\ssf\rangle
\qquad\forall\;I\!\subset\msb R^+,
\end{equation*}
\vspace{-5mm}

we have
\vspace{-1mm}
\begin{equation*}
\boldpi(\partial)\msf h^{\vsf n+|R^+\ssb|}
=\sum\nolimits_{\msf k=1}^{\ssf|\vsb R^+\ssb|}\ssf
(n\msb+\msb|R^+\vsb|)\ssf\cdots\ssf(n\msb+\msb|R^+\vsb|\!-\msb k\msb+\!1)
\,h^{\vsf n+|R^+\ssb|-k}\ssf f_k\msf,
\end{equation*}
where
\begin{equation}\label{DefinitionFk}
f_k=\!\sum_{\substack{
I_1\sqcup\ssf I_2\ssf\sqcup\ssf\dots\ssf\sqcup\ssf I_k=R^+\\I_1\vsb\ne\ssf\emptyset,\ssf\dots\vsf,\vsf I_k\ne\ssf\emptyset}}\!
\bigl[\vsf\pi_{I_1}\ssb(\partial)(h\msb+\msb2)\bigr]\cdots\bigl[\vsf\pi_{I_k}\ssb(\partial)(h\msb+\msb2)\bigr]
\end{equation}
is an exponential polynomial,
which belongs to the \ssf$\Z$\ssf-\msf span of \ssf$\{\ssf e^\lambda\,|\,\lambda\!\in\!P\ssf\}$\ssf.
As $f_k$ is skew, it is divisible by \ssf$\Delta$ \ssf
and the quotient \ssf$g_k$ is a symmetric exponential polynomial of degree $\le\msb k\msb-\msb r$\vsf.
Here, the \textit{degree\/} of an exponential polynomial
\begin{equation*}
g=\sum\nolimits_{\lambda\in P}g(\lambda)\msf e^{\ssf\lambda}
\end{equation*}
is defined by
\begin{equation*}
\deg g=\begin{cases}
\,\max\ssf\{\ssf|\lambda|\,|\msf g(\lambda)\ssb\ne\ssb0\ssf\}\msb\in\ssb\N
&\text{if \,}g\msb\ne\ssb0\ssf,\\
\,-\infty
&\text{if \,}g\msb=\ssb0\ssf.
\end{cases}\end{equation*}
We deduce in particular that \ssf$g_k$ vanishes when \ssf$k\msb<\msb r$ \ssf and it remains for us to show that
\ssf$g_k$ is proportional to \ssf$(h\ssb+\ssb2)^{k-r}$ when \ssf$r\ssb\le\ssb k\ssb\le\ssb|R^+\vsb|$\ssf.
For this purpose, notice that,
for every positive root \ssf$\alpha\ssb=\ssb e_j\msb-\ssb e_k$ ($1\!\le\ssb j\msb<\ssb k\ssb\le\ssb r\msb+\!1$)
and for every weight \vsf$\langle\ssf\lambda\vsf,\vsb z\ssf\rangle\msb=\ssb z_{\vsf\ell}$ ($1\!\le\ssb\ell\ssb\le\ssb r\msb+\!1$)
in $W_0.\vsf\lambda_1$, we have
\begin{equation}\label{AuxiliaryEquation3}
\langle\ssf\alpha,\lambda\ssf\rangle=\begin{cases}
\,1&\text{when \,}\ell\ssb=\ssb j\ssf,\\
\,-1&\text{when \,}\ell\ssb=\ssb k\msf,\\
\,0&\text{otherwise.}
\end{cases}\end{equation}

Consequently,
\begin{description}[labelindent=0pt,labelwidth=4mm,labelsep*=2pt,leftmargin=!]
\item[\rm(a)]
for every \ssf$\alpha\msb\in\msb R^+$\ssb, there are exactly two weights \ssf$\lambda\msb\in\msb W_0.\vsf\lambda_1$
such that \ssf$\langle\ssf\alpha,\lambda\ssf\rangle\msb\ne\ssb0\ssf$,
\item[\rm(b)]
for every \ssf$\lambda\msb\in\msb W_0.\vsf\lambda_1$\ssf,
there are exactly \vsf$r$ roots \ssf$\alpha\msb\in\msb R^+$
such that \ssf$\langle\ssf\alpha,\lambda\ssf\rangle\msb\ne\ssb0$\ssf.
\end{description}
On the one hand, we deduce from \eqref{ProductFormulaAr} and \eqref{AuxiliaryEquation3} that
\begin{equation*}
\tfrac{\partial_\alpha[\ssf h(z)\vsf+\ssf2\ssf]}{h(z)\vsf+\ssf2}
=\tfrac{e^{\ssf z_j}}{e^{\ssf z_j}+\vsf1}-\tfrac{e^{\ssf z_k}}{e^{\ssf z_k}+\vsf1}
=e^{\frac{z_j+\vsf z_k}2}\tfrac{e^{\frac\alpha2}-\ssf e^{-\frac\alpha2}}{(e^{\ssf z_j}+\vsf1)(e^{\ssf z_k}+\vsf1)}
\qquad(\alpha\ssb=\ssb e_j\msb-\ssb e_k)
\end{equation*}
and obtain this way the leading term
\begin{equation*}
f_{\ssf|\vsb R^+\vsb|}
=\prod\nolimits_{\ssf\alpha\in R^+}\msb\partial_\alpha(h\msb+\msb2)
=(h+2)^{|\vsb R^+\vsb|-r}\ssf\Delta\,.
\end{equation*}
On the other hand, we deduce from (b) above that each term
on the right hand side of \eqref{DefinitionFk} is divisible by $(e^{\ssf\lambda}\msb+\msb1)^{k-r}$,
for every \ssf$\lambda\msb\in\msb W_0.\vsf\lambda_1$\vsf,
and hence by $(h\msb+\msb2)^{k-r}$.
In summary, $f_k$ is divisible by and hence proportional to $(h\msb+\msb2)^{k-r}\ssf\Delta$\ssf.
This concludes the proof of \eqref{DifferentiationFormula1Ar}
and \eqref{DifferentiationFormula2Ar} is proved similarly.
\end{proof}

Let us next generalize partially Lemma \ref{PropertiesShiftA2}

\begin{lemma}\label{PropertiesShiftAr}
\begin{description}[labelindent=0pt,labelwidth=5mm,labelsep*=1pt,leftmargin=!]
\item[\rm(a)]
The function \,$\Phi$ is strictly convex,
tends to \ssf$+\infty$ at infinity
and reaches its minimum \ssf$\phi(\delta)\!\in\!(-\ssb\log N\ssb,0\ssf]$
at a single point \ssf$s\!\in\!\cl{\sector}$,
which is the unique solution to the equation \,$\tfrac{dh(s)}{h(s)}\msb=\ssb\delta$\vsf.
Moreover \,$s\ssb\longleftrightarrow\ssb\delta$ is an analytic bijection between \msf$\smash{\cl{\sector}}$ \msb and \,$\{\ssf\delta\msb\in\msb\smash{\cl{\sector}}\ssf|\msf|\delta|\!<\!1\vsf\}$\ssf.
\end{description}
\item[\rm(b)]
The following properties hold\,$:$
\begin{description}[labelindent=4pt,labelwidth=4mm,labelsep*=1pt,leftmargin=!]
\item[\textbullet]
\,$s\ssb=\ssb0$ \ssf$\Longleftrightarrow$ \ssf$\delta\ssb=\ssb0$\ssf,
\item[\textbullet]
\,$h(s)\ssb\approx\vsb\prod_{1\le\vsf j\le\vsf r+1}\bigl(1\msb\vee\msb e^{\ssf s_j}\ssb\bigr)\msb
=\ssb\smash{e^{\ssf\sum_{1\le j\le r+1}\msb0\ssf\vee\vsb s_j}}\vphantom{\frac||}$,
\item[\textbullet]
\,$\smash{\tfrac2{h(s)}}\ssb\le\ssb1\!-\msb|\delta|$\ssf.
\end{description}
\end{lemma}

\begin{proof}
(a) and the first claim in (b) are proved as in Section \ref{sectionA2}.
The second claim in (b) follows from \eqref{ProductFormulaAr} and more precisely from
\begin{equation}\label{AuxiliaryResult1}
h(s)\vsb\approx h(s)\msb+\ssb2=\prod\nolimits_{\vsf1\le\vsf j\le\vsf r+1}\bigl(1\!+\msb e^{\ssf s_j}\bigr)\ssf.
\end{equation}
Let us turn to the third claim in (b).
We use \eqref{ProductFormulaAr} and actually the right hand side of \eqref{AuxiliaryResult1} to compute the logarithmic derivative
\begin{equation*}
\tfrac{d\vsf h(s)}{h(s)\ssf+\ssf2}
=\sum\nolimits_{\vsf1\le j\le r+1}\msb\tfrac{e^{\ssf s_{\vsb j}}}{e^{\ssf s_{\vsb j}}+\ssf1}\msf e_j
-\tfrac1{r+1}\ssf\Bigl(\ssf\sum\nolimits_{\vsf1\le k\le r+1}\msb\tfrac{e^{\ssf s_{\vsb k}}}{e^{\ssf s_{\vsb k}}+\ssf1}\Bigr)
\Bigl(\ssf\sum\nolimits_{\vsf1\le k\le r+1}\msb e_k\Bigr)\ssf.
\end{equation*}
As \msf$\delta\ssb=\ssb\tfrac{dh(s)}{h(s)}$\vsf, we deduce that
\begin{equation*}
|\delta|=\sum\nolimits_{\vsf1\le j\le r}\ssb\langle\ssf\alpha_j,\delta\ssf\rangle=\delta_1\!-\ssb\delta_{r+1}
=\tfrac{h(s)\ssf+\ssf2}{h(s)}\ssf\tfrac{e^{\ssf s_{\vsb1}}-\ssf e^{\ssf s_{r+1}}}{(e^{\ssf s_{\vsb1}}+\ssf1)(e^{\ssf s_{r+1}}+\ssf1)}\,,
\end{equation*}
hence \msf$h(s)\ssf(1\!-\msb|\delta|\vsf)\msb=\ssb\varpi(s)\msb-\ssb2$\ssf, where
\begin{align}
\varpi(s)
&=\ssb\bigl[h(s)\msb+\msb2\vsf\bigr]
\Bigl[1\msb-\msb\tfrac{e^{\ssf s_{\vsb1}}-\ssf e^{\ssf s_{r+1}}}{(e^{\ssf s_{\vsb1}}+\ssf1)(e^{\ssf s_{r+1}}+\ssf1)}\Bigr]
\nonumber\\
&=\bigl(\vsf e^{\ssf s_1\vsb+\vsf s_{r+1}}\!+\ssb2\ssf e^{s_{r+1}}\!+\msb1\bigr)\ssf
\prod\nolimits_{\vsf1<j<r+1}\bigl(\vsf e^{\ssf s_{\vsb j}}\!+\!1\bigr)\ssf.
\label{Product}
\end{align}
By expanding the product \eqref{Product}, we obtain a sum of positive terms including
\begin{align*}
e^{\ssf s_1\vsb+\vsf s_{r+1}}\ssb\times\prod\nolimits_{\vsf1<j<r+1}\ssb e^{\ssf s_{\vsb j}}&=\ssf1\msf,\\
1\times\prod\nolimits_{\vsf1<j<r+1}\ssb1\ssf&=\ssf1\msf,\\
e^{\ssf s_1\vsb+\vsf s_{r+1}}\ssb\times\prod\nolimits_{\vsf1<j<r+1}\ssb1\ssf&=\msf e^{\ssf s_1\vsb+\vsf s_{r+1}}\ssf,\\
1\times\prod\nolimits_{\vsf1<j<r+1}\ssb e^{\ssf s_{\vsb j}}&=\msf e^{-s_1\vsb-\vsf s_{r+1}}\ssf.
\end{align*}
Thus \,$\varpi(s)\msb\ge\ssb2\msb+\msb2\vsb\cosh\ssf(s_1\!+\msb s_{\vsf r+1})\msb\ge\ssb4$ \msf
and we conclude that
\begin{equation*}
h(s)\ssf(1\!-\msb|\delta|\vsf)\msb=\ssb\varpi(s)\msb-\ssb2\ssb\ge\ssb2\msf.
\end{equation*}
\end{proof}

\subsection{\bf Proof of \eqref{OptimalBoundAr} for $n$ large.}

The inversion formula yields first
\begin{equation}\label{HeatFormula3}
p_{\ssf n}(x)=C_0\,\sigma^n\ssf q^{-\langle\ssf\rho,\ssf x^+\rangle}
\int_{\ssf U}\ssb h(i\ssf\theta)^n\,\tfrac
{\Delta(i\ssf\theta)\,e^{-i\langle x^+\ssb+\rho,\ssf\theta\ssf\rangle}}
{\mathbf{b}(i\ssf\theta)}\,d\ssf\theta
\end{equation}
and next
\begin{equation}\label{HeatFormula4}
p_{\ssf n}(x)
=C_0\,\sigma^{\ssf n}\,q^{-\langle\ssf\rho,\ssf x^+\rangle}\ssf
e^{-\langle\ssf x^+\ssb+\ssf\rho,\ssf s\ssf\rangle}\ssb
\int_{\ssf U}\ssb h(s\ssb+\ssb i\ssf\theta)^n\,\tfrac
{\Delta(s\ssf+\ssf i\ssf\theta)\,e^{-i\langle x^+\ssb+\rho,\ssf\theta\ssf\rangle}}
{\mathbf{b}(s\ssf+\ssf i\ssf\theta)}\,d\ssf\theta
\end{equation}
after a shift of contour.
Notice that \vsf$s$ \vsf remains bounded, under the present assumption \ssf$|x|\ssb\le\ssb(1\!-\msb\eta)\ssf n$\ssf,
and that the function
\begin{equation*}
\Psi(\theta)=\ssf\log\tfrac{h(s\ssf+\ssf i\ssf\theta)}{h(s)}-i\ssf\langle\ssf\delta,\theta\ssf\rangle
\end{equation*}
satisfies
\begin{equation}\label{AuxiliaryEquation4}
-\Re\Psi(\theta)\approx|\theta|^2
\quad\text{and}\quad
|\Im\Psi(\theta)|\lesssim|\theta|^3
\end{equation}
in a neighborborhood of the origin, uniformly in $s$ and $\delta$.
All these properties are established as in Section \ref{sectionA2}.
With the notation of Lemma \ref{LemmaAr}, let us first replace \msf$h^{\vsf n}\Delta$ \ssf by
\begin{equation*}
r_{n+r-|\vsb R^+\ssb|}(h)\msf h^{n+r-|\vsb R^+\ssb|}\msf\Delta
=\tfrac{(n+r-|\vsb R^+\ssb|)\ssf!}{(n+r)\ssf!}\,\boldpi(\partial)\ssf h^{\vsf n+r}
\end{equation*}
in \eqref{HeatFormula3} or \eqref{HeatFormula4},
and let us denote by \ssf$\widetilde{p}_{\ssf n}(x)$ the resulting expression.
Then, after an integration by parts, we obtain the desired estimate \eqref{OptimalBoundAr}
for \ssf$|\ssf\widetilde{p}_{\ssf n}(x)|$ \ssf as we did for \eqref{DefinitionJtildenx} in Subsection \ref{Proof3}.
For any \ssf$r\ssb\le\ssb k\ssb\le\ssb|R^+\ssb|$\ssf,
let us next replace \ssf$h^{\vsf n}$ by \ssf$h^{\ssf n+r-k}\ssf(h\msb+\msb2)^{k-r}$
in \eqref{HeatFormula3} or \eqref{HeatFormula4},
and let us denote by \ssf$\widetilde{p}_{\ssf n,\ssf k}(x)$ the resulting expression.
By applying the binomial formula to powers of \ssf$h\msb+\msb2$\ssf, we get
\begin{equation*}\label{ChainInequality}
p_{\ssf n}(x)\ssb=\widetilde{p}_{\ssf n,\ssf r}(x)\ssb\le\ssf\dots\ssf
\le\widetilde{p}_{\ssf n,\ssf k}(x)\ssb\le\widetilde{p}_{\ssf n,\ssf k+1}(x)\ssb
\le\ssf\dots\ssf\le\widetilde{p}_{\ssf n,\ssf|R^+\ssb|}(x)\msf.
\end{equation*}
With the notation of Lemma \ref{LemmaAr}, assume that $n$ is large enough so that
\begin{equation*}
\sum\nolimits_{\ssf r\le k<|R^+|}|c_k|\msf\tfrac{(n+r-|\vsb R^+\ssb|)\ssf!}{(n+r-k)\ssf!}\,\widetilde{p}_{\ssf n,\ssf k}(x)
\le\tfrac12\,\widetilde{p}_{\ssf n,\ssf|R^+\ssb|}(x)\msf.
\end{equation*}
Then
\begin{equation*}
\widetilde{p}_{\ssf n}(x)\ssb=\widetilde{p}_{\ssf n,\ssf|R^+\ssb|}(x)
+\sum\nolimits_{\ssf r\le k<|R^+|}c_k\msf\tfrac{(n+r-|\vsb R^+\ssb|)\ssf!}{(n+r-k)\ssf!}\,\widetilde{p}_{\ssf n,\ssf k}(x)
\ge\tfrac12\,\widetilde{p}_{\ssf n,\ssf|R^+\ssb|}(x)
\end{equation*}
is nonnegative with leading term \ssf$\widetilde{p}_{\ssf n,\ssf|R^+\ssb|}(x)$\ssf.
Hence \ssf$\widetilde{p}_{\ssf n,\ssf|R^+\ssb|}(x)$ and consequently \ssf$\widetilde{p}_{\ssf n}(x)$
satisfy \eqref{OptimalBoundAr}.
\hfill$\square$
\medskip

\subsection{Proof of \eqref{UpperBoundAr} for $n$ large.}
\label{ProofUpperEstimateAr}

Let us replace \ssf$h^{\vsf n}$ by $(h\msb+\msb2)^n$ in \eqref{HeatFormula3} and \eqref{HeatFormula4}
and let us denote by $\widetilde{p}_{\ssf n}(x)$ the resulting kernel.
On the one hand, we deduce again from the binomial formula that
\ssf$\widetilde{p}_{\ssf n}\msb\ge\ssb p_{\ssf n}$\vsf.
On the other hand, by performing an integration by parts based on \eqref{DifferentiationFormula2Ar}
and by resuming our overall strategy, we obtain
\begin{equation*}
\widetilde{p}_{\ssf n}(x)\lesssim
n^{-\frac r2-\vsf|R^+|}\,\widetilde{\boldsigma}^{\ssf n}\ssf F_0(x)\,
\tfrac{e^{\ssf n\vsf\widetilde{\phi}(\delta)}}{\prod_{\ssf\alpha\in\vsb R^+}\!\sqrt{1-\vsf\langle\alpha,\ssf\delta\ssf\rangle}}\,,
\end{equation*}
where \ssf$\widetilde{\boldsigma}\ssb=\ssb\sigma\ssf(N\hspace{-.75mm}+\ssb2)$\vsf, 
$\widetilde{\Phi}(z)\msb=\ssb\log\tfrac{h(z)\ssf+\ssf2}{N+\ssf2}\ssb-\ssb\langle\ssf\delta,z\ssf\rangle$
and \ssf$\widetilde{\phi}(\delta)\msb=\ssb\min_{\ssf z\vsf\in\vsf\apartment}\widetilde{\Phi}(z)$\ssf.
In order to conclude, let us compare the expressions
\ssf$\smash{e^{\ssf n\ssf\widetilde{\phi}(\delta)}}$ and \ssf$e^{\ssf n\ssf\phi(\delta)}$.
It follows from
\begin{equation*}
0\le\bigl[\ssf\widetilde{\Phi}(z)\msb+\ssb\log\ssf(N\hspace{-.75mm}+\ssb2)\vsf\bigr]
-\bigl[\ssf\Phi(z)\msb+\ssb\log N\ssf\bigr]
=\log\ssf\bigl[1\!+\msb\tfrac2{h(z)}\bigr]\le\tfrac2{h(z)}
\qquad\forall\;z\msb\in\msb\apartment\ssf,
\end{equation*}
that
\begin{equation*}
0\le\bigl[\ssf\widetilde{\phi}(\delta)\msb+\ssb\log\ssf(N\hspace{-.75mm}+\ssb2)\vsf\bigr]
-\bigl[\ssf\phi(\delta)\msb+\ssb\log N\ssf\bigr]\le\tfrac2{h(s)}\,,
\end{equation*}
which is $\le\msb1\!-\msb|\delta|$\ssf, according to the last inequality in Lemma  \ref{PropertiesShiftAr}.(b).
Thus
\begin{equation*}
\widetilde{\boldsigma}^n\msf e^{\ssf n\ssf\widetilde{\phi}(\delta)}\ssb
\le\boldsigma^{\vsf n}\msf e^{\ssf n\ssf\phi(\delta)}\msf e^{\ssf n\ssf(1-\vsf|\delta|)}
\end{equation*}
and this concludes the proof of \eqref{UpperBoundAr}.\hfill$\qed$
\medskip

This concludes the proof of Theorem \ref{TheoremHeatEstimateAr}.

\appendix
\section{Some formulae}
\label{Appendix}

This appendix is devoted to the following remarkable formulae in rank 2.
Consider
\vspace{-1.5mm}
\begin{equation*}
h\ssf=\ssf c_1\overbrace{(\ssf
e^{\ssf\lambda_1}\!+e^{-\lambda_2}\!+e^{\ssf\lambda_2-\lambda_1}
)}^{h_1}\ssf+\;\ctwo\overbrace{(\ssf
e^{\ssf\lambda_2}\!+e^{-\lambda_1}\!+e^{\ssf\lambda_1-\lambda_2}
)}^{h_2}\,,
\end{equation*}
\vspace{-3mm}

where \ssf$\cone$ and \ssf$\ctwo$ are arbitrary real constants.

\begin{lemma}\label{GeneralDifferentiationLemmaRankTwo}
For every integer \,$n\!\in\!\N$\ssf,
\begin{equation}\label{GeneralDifferentiationFormulaRankTwo}
\boldpi(\partial)\,h^n
=\cone\ssf\ctwo\,n^2\ssf(n\!-\!1)\,h^{n-2}\ssf\Delta
+(\cone^{\ssf3}\!+\ssb \ctwo^{\ssf3}\ssf)\,
n\ssf(n\!-\!1)\ssf(n\!-\!2)\,h^{\ssf n-3}\ssf\Delta\,.
\end{equation}
\end{lemma}

\begin{proof}
We have
\vspace{-1mm}
\begin{align*}
\smash{\partial_{\alpha_1}\partial_{\alpha_2}\partial_{\rho}}\,h^n
&=n\,\smash{h^{n-1}}\ssf A\\
&\ssf+n\ssf(n\!-\!1)\,\smash{h^{n-2}}\ssf B\\
&\ssf+n\ssf(n\!-\!1)\ssf(n\!-\!2)\,\smash{h^{n-3}}\,C\,,
\end{align*}
\vspace{-5mm}

where
\vspace{-.5mm}
\begin{equation*}\begin{cases}
\,A=\partial_{\alpha_1}\partial_{\alpha_2}\partial_{\rho\ssf}h\,,\\
\,B=(\partial_{\alpha_1}h)\ssf(\partial_{\alpha_2}\partial_{\rho\ssf}h)\ssb
+\ssb(\partial_{\alpha_2}h)\ssf(\partial_{\alpha_1}\partial_{\rho\ssf}h)\ssb
+\ssb(\partial_{\rho\ssf}h)\ssf(\partial_{\alpha_1}\partial_{\alpha_2}h)\,,\\
\,C=(\partial_{\alpha_1}h)\ssf(\partial_{\alpha_2}h)\ssf(\partial_{\rho\ssf}h)\,.
\end{cases}\end{equation*}
Elementary computations yield first
\begin{align*}
&\begin{aligned}\partial_{\alpha_1}h
&=\smash{\cone\bigl(\ssf e^{\ssf\lambda_1}\ssb-e^{\ssf\lambda_2-\lambda_1}\bigr)
+\ctwo\ssf\bigl(\ssf e^{\ssf\lambda_1-\lambda_2}\ssb-e^{-\lambda_1}\bigr)}\\
&=\smash{\bigl(\ssf e^{\frac{\alpha_1}2}\!-\ssb e^{-\frac{\alpha_1}2}\bigr)
\bigl(\ssf\cone\ssf e^{\frac{\lambda_2}2}\!+\ctwo\,e^{-\frac{\lambda_2}2}\bigr)}
\,,\end{aligned}\\
&\begin{aligned}\partial_{\alpha_2}h
&=\smash{\cone\bigl(\ssf e^{\ssf\lambda_2-\lambda_1}-e^{-\lambda_2}\bigr)
+\ctwo\ssf\bigl(\ssf e^{\ssf\lambda_2}\!-e^{\ssf\lambda_1-\lambda_2}\bigr)}\\
&=\smash{\bigl(\ssf e^{\frac{\alpha_2}2}\!-\ssb e^{-\frac{\alpha_2}2}\bigr)
\bigl(\ssf\cone\ssf e^{-\frac{\lambda_1}2}\!+\ctwo\,e^{\frac{\lambda_1}2}\bigr)}
\,,\end{aligned}\\
&\begin{aligned}\partial_{\rho\ssf}h
&=\smash{\cone\bigl(\ssf e^{\ssf\lambda_1}\ssb-e^{-\lambda_2}\bigr)
+\ctwo\ssf\bigl(\ssf e^{\ssf\lambda_2}\ssb-e^{-\lambda_1}\bigr)}\\
&=\smash{\bigl(\ssf e^{\ssf\frac\rho2}\!-\ssb e^{-\frac\rho2}\bigr)
\bigl(\ssf\cone\ssf e^{\frac{\lambda_1-\lambda_2}2}\!+\ctwo\,e^{\frac{\lambda_2-\lambda_1}2}\bigr)}
\,,\end{aligned}\\
&\partial_{\alpha_2}\partial_{\rho\ssf}h
=\cone\ssf e^{-\lambda_2}\ssb
+\ctwo\,e^{\ssf\lambda_2}\ssf,\\
&\partial_{\alpha_1}\partial_{\rho\ssf}h
=\cone\ssf e^{\ssf\lambda_1}\ssb
+\ctwo\,e^{-\lambda_1}\ssf,\\
&\partial_{\alpha_1}\partial_{\alpha_2}h
=-\,\cone\ssf e^{\ssf\lambda_2-\lambda_1}\ssb
-\ctwo\,e^{\ssf\lambda_1-\lambda_2}\ssf,
\end{align*}
\vspace{-5mm}

and next
\vspace{-.5mm}
\begin{equation*}\begin{cases}
\,A=0\,,\\
\,B=2\,\cone\ssf\ctwo\,\Delta\,,\\
\,C=\bigl(\ssf\cone\ssf\ctwo\,h
+\cone^{\ssf3}\ssb+\ctwo^{\ssf3}\ssf\bigr)\ssf\Delta\,.
\end{cases}\end{equation*}
\vspace{-1.5mm}

This concludes the proof of Lemma \ref{GeneralDifferentiationLemmaRankTwo}.
\end{proof}

\begin{remark}
When \,$\cone$ or \,$\ctwo$ \ssf is equal to \,$0$\ssf,
notice that \eqref{GeneralDifferentiationFormulaRankTwo} reduces to
\begin{equation*}
\boldpi(\partial)\,h_j^n\ssb
=n\ssf(n\!-\!1)\ssf(n\!-\!2)\,h_j^{n-3}\ssf\Delta
\qquad(\ssf j\ssb=\ssb1,2\ssf)\ssf.
\end{equation*}
\end{remark}

\begin{lemma}\label{RemarkableLemmaRankTwo}
Assume that \,$\cone$ \ssb and \,$\ctwo$ \ssf are nonzero.
Then the following product and differentiation formulae hold for
\,$\widetilde{h}\ssb=\ssb\cone\ssf\ctwo\,h\ssb
+\ssb\cone^{\ssf3}\hspace{-.5mm}+\ssb\ctwo^{\ssf3}$\;{\rm:}
\begin{gather}
\widetilde{h}\ssf=
\bigl(\ssf\ctwo\,e^{\frac{\lambda_1}2}\hspace{-.75mm}
+\ssb\cone\,e^{-\frac{\lambda_1}2}\bigr)
\bigl(\ssf\ctwo\,e^{-\frac{\lambda_2}2}\hspace{-.75mm}
+\ssb\cone\,e^{\frac{\lambda_2}2}\bigr)
\bigl(\ssf\ctwo\,e^{\frac{\lambda_2-\lambda_1}2}\hspace{-.75mm}
+\ssb\cone\,e^{\frac{\lambda_1-\lambda_2}2}\bigr)\ssf,
\label{ProductFormulaA2RankTwo}\\
\boldpi(\partial)\,\widetilde{h}^{\ssf n}\ssb
=\cone^{\ssf3}\ssf\ctwo^{\ssf3}\hspace{.75mm}
n^2\ssf(n\!-\!1)\,\widetilde{h}^{\ssf n-2}\ssf\Delta\,.
\label{DifferentiationFormulaRankTwo}\end{gather}
\end{lemma}

\begin{proof}
The proof of \eqref{ProductFormulaA2RankTwo} is straightforward
and \eqref{DifferentiationFormulaRankTwo} is proved
as \eqref{GeneralDifferentiationFormulaRankTwo}.
\end{proof}

\end{document}